\documentclass{amsproc}
\usepackage{euscript, graphicx, epstopdf}
\usepackage{cases}
\usepackage{mathrsfs}
\usepackage{bbm}
\usepackage{amssymb}
\usepackage{txfonts}
\usepackage{amscd}
\usepackage{amsfonts,latexsym,amsmath,amsxtra,mathdots,amssymb,latexsym,mathabx}
\usepackage[all,cmtip]{xy}
\usepackage{color}
\usepackage{colordvi}
\usepackage{multicol}
\usepackage{hyperref}
\usepackage{tikz}
\usepackage{float}
\usepackage{setspace, floatflt}

\allowdisplaybreaks

\newcommand{\BA}{{\mathbb {A}}} 
\newcommand{\BC}{{\mathbb {C}}}

\newcommand{\BQ}{{\mathbb {Q}}} \newcommand{\BR}{{\mathbb {R}}}

\newcommand{\GL}{{\mathrm {GL}}} \newcommand{\PGL}{{\mathrm {PGL}}}
\newcommand{\SL}{{\mathrm {SL}}}

\def\-{^{-1}}

\def\shskip{\hskip 0.5 pt}

\def\dx{d^{\times} \hskip -1 pt}

\def\Fv{F_{\varv}}
\def\tv{\textit{v}}
\def\ro{\mathrm{o}}
\def\tpi{\widetilde{\pi}}

\makeatletter
\g@addto@macro\normalsize{\setlength\abovedisplayskip{3pt}}
\makeatother

\makeatletter
\g@addto@macro\normalsize{\setlength\belowdisplayskip{3pt}}
\makeatother

\newcommand{\delete}[1]{}

\newenvironment{myquote}%
{\list{}{\leftmargin=0.3in\rightmargin=0.3in}\item[]}%
{\endlist}

\theoremstyle{plain}
\newtheorem*{prop*}{Proposition}
\newtheorem{thm}{Theorem}[section] \newtheorem{cor}[thm]{Corollary}
\newtheorem{lem}[thm]{Lemma}  \newtheorem{prop}[thm]{Proposition}

\newtheorem {conj}[thm]{Conjecture} 

\newtheorem {rem}[thm]{Remark}

\numberwithin{equation}{section}
\newtheorem*{acknowledgement}{Acknowledgements}

\newcommand{\red}[1]{\textcolor{red}{#1}}

\begin{document}

	\title[The Waldspurger Formula over Number Fields]{On the Waldspurger Formula and the Metaplectic Ramanujan Conjecture over Number Fields}

	\author{Jingsong Chai}
	\address{School of Mathematics\\ Sun Yat-sen University \\Guangzhou,  510275\\China}
	\email{chaijings@mail.sysu.edu.cn}

	\author{Zhi Qi}
	\address{School of Mathematical Sciences\\ Zhejiang University\\Hangzhou 310027\\China}
	\email{zhi.qi@zju.edu.cn}

%	School of Mathematical Science, Zhejiang University, Hangzhou 310027, Zhejiang,
	
	%\keywords{Waldspurger formula, Waldspurger correspondence, relative trace formula, Ramanujan conjecture}
	\thanks{The first author is supported by the National Natural Science
Foundation of China [Grant 11771131].}

	\begin{abstract}
		In this paper, by inputting the Bessel identities over the complex field in previous work of the authors, the Waldspurger formula of Baruch and Mao is extended from totally real fields to arbitrary number fields. This is applied to give a non-trivial bound towards the Ramanujan conjecture for automorphic forms on the metaplectic group $\widetilde{\mathrm{SL}}_2$ for the first time in the generality of arbitrary number fields.
	\end{abstract}
	
	\maketitle
	
	\section{Introduction}

	\subsection{Backgrounds}

	In the work \cite{Waldspurger-Formula}, Waldspurger proved his celebrated formula connecting the Fourier coefficients of a  modular form of  half integral weight to the twisted central $L$-values of a modular form of integral weight. The formula in \cite{Waldspurger-Formula} is over the rational field $\BQ$.  %Waldspurger's formula was later extended and made more explicit by many authors.
	
	In the work \cite{BaruchMao-Global}, Baruch and Mao proved a very explicit Waldspurger-type formula  for automorphic forms over {\it totally real fields}.
	Waldspurger's formula was also extended by many other authors in various cases. An incomplete list (in the chronological order) includes the papers of Kohnen-Zagier \cite{Wald-Kohnen-Zagier81,Wald-Kohnen-82}, Niwa \cite{Wald-Niwa82}, Gross \cite{Wald-Gross85},  Shimura \cite{Wald-Shimura93}, Katok-Sarnak \cite{Wald-Katok-Sarnak93}, Khuri-Makdisi \cite{Wald-KM96}, Kojima \cite{Wald-Kojima1-99}--\cite{Wald-Kojima3-13}, \nocite{Wald-Kojima2-00,Wald-Kojima3-04} Prasanna \cite{Wald-Prasanna-00} and Altu\u g-Tsimerman \cite{Wald-Fp[T]14} (the function field case).
	
	The Waldspurger-type formula of Baruch and Mao in  \cite{BaruchMao-Global} may be regarded as the ultimate version of its kind---if there were no totally real restriction on the ground field---both by removing  all hypotheses assumed by earlier authors as well as by making Waldspurger's formula entirely explicit. For example, after the translation into the classical language (over the rational field $\BQ$), the formula of Baruch and Mao is a generalization of  the Kohnen-Zagier formula in \cite{Wald-Kohnen-Zagier81} for all fundamental discriminants $D$  without
	any restriction. It was later applied in \cite{Baruch-Mao-KZ-Maass} and \cite{HI-Kohnen-Hilbert} to obtain the generalized Kohnen-Zagier formula for Maass forms and Hilbert modular forms respectively.

	Applications of the Waldspurger-type formula in \cite{BaruchMao-Global} include:
	\begin{itemize}
		\item [(1).]  The equivalence of the Ramanujan conjecture for half integral weight
		forms  with  the Lindel\"of hypothesis for twisted central $L$-values of integral weight forms, in the generality of totally real number fields;
		\item [(2).] An efficient method of computing the central value of the twisted $L$-functions
		associated to the elliptic curve $X_0(11)$, by the aforementioned generalization
		of the Kohnen-Zagier formula.
	\end{itemize}
%Note that (1)
A further application in the work of Cogdell, Piatetski-Shapiro
and Sarnak \cite{CPSS-Hilbert} (see also Blomer and Harcos \cite{Blomer-Harcos-TR}) is:
\begin{itemize}
	\item [(3).] A solution for the last open case of  Hilbert's eleventh problem on representing integers by positive definite integral {ternary} quadratic forms
	over {totally real fields}.
\end{itemize}
This is a consequence of the Ramanujan-Lindel\"of equivalence in (1) combined with the subconvexity bound for twisted $L$-values for Hilbert modular forms over totally real fields in \cite{CPSS-Hilbert} or \cite{Blomer-Harcos-TR}. See \cite{Cogdell-Hilbert} for an account of their solution of Hilbert's eleventh problem and related references.
	
	\subsection{Motivation}

	The Waldspurger formula of Baruch and Mao in \cite{BaruchMao-Global} is a consequence of the local spectral theory from the viewpoint of {Bessel identities} over non-Archimedean fields and the real field in \cite{BaruchMao-NA, BaruchMao-Real} and the global theory for a relative trace formula of Jacquet  in \cite{Jacquet-RTF}. These are incorporated nicely in the framework of Waldspurger \cite{Waldspurger-Shimura,Waldspurger-Shimura3} on the representation theoretic form of the Shimura correspondence \cite{Shimura-Annals}.
	
	In Baruch and Mao's project, the proof of the local Bessel identities is the foundational and technical part.
%When the local field
	
	It has long been known  that the Bessel functions for $\GL_2 (\BR)$ (and  $\widetilde{\SL_2} (\BR)$), defined over $\BR \smallsetminus \{0\} = \BR_+ \cup - \BR_+$, may be expressed in terms of classical Bessel functions on $\BR_+$.\footnote{The first precise interpretation of  classical Bessel functions in representation theory is due to Cogdell and Piatetski-Shapiro \cite{CPS} (1990). The occurrence of Bessel functions in number theory however may be traced back to the formulae of Vorono\"i   \cite{Voronoi} (1904), Petersson \cite{Petersson} (1932) and Kuznetsov \cite{Kuznetsov} (1981). See \cite{IK}.} %, which arose in the summation formula of \Voronoi and the trace formula of Petersson and Kuznetsov.
	The Bessel identities over $\BR$ in \cite{BaruchMao-Real} can be regarded as the representation theoretic interpretation of the classical exponential integral formulae of Weber and Hardy on the Fourier transform of Bessel functions on $\BR_+$.
	
	The Bessel functions for $\PGL_2 (\BC)$, expressed in terms of classical Bessel functions on $\BC \smallsetminus \{0\}$, were first discovered by Bruggemann and Motohashi in 2003 \cite{B-Mo}, and later by Lokvenec-Guleska \cite{B-Mo2} in 2004 for $\SL_2 (\BC)$.\footnote{The {\it spherical} Bessel functions for $\SL_2 (\BC)$ actually appeared much earlier in \cite{M-W-Kuz} (1990). The representation theory of Bessel functions for $\GL_2 (\BC)$ may be found in \cite{B-Mo-Kernel2,Mo-Kernel2}, \cite{Baruch-Kernel} and \cite[Chapter 4]{Qi-Bessel}.} The discovery was not long before the publication of Baruch and Mao's work, and the complex analogue of the formulae of  Weber and Hardy was not available at that time.
	There is however a remark in \cite{BaruchMao-Global}:
	\begin{myquote}
{\it Our {\rm(}Waldspurger-type{\rm)} formula can be extended to all number fields once
	the local result in \cite{BaruchMao-NA} and \cite{BaruchMao-Real} is extended to the case of the complex field}.
	\end{myquote}

Recently, tempted by its outcome indicated in this remark, the authors had a series of papers towards the Bessel identities over $\BC$. It is now established in \cite{Chai-Qi-Bessel} by the classical-like Weber-Hardy type exponential integral formula  in \cite{Qi-Sph,Qi-II-G} for the Fourier transform of Bessel functions for $\PGL_2 (\BC) $ (see Remark \ref{rmk:Local Bessel Identity} for more discussions).  The present article is the final payment for these works---the Waldspurger formula over arbitrary number fields.

As an immediate application, we shall extend the Ramanujan-Lindel\"of equivalence in (1) to   arbitrary number fields. Moreover, combined with the $\GL_2 \times \GL_1$ subconvexity results over number fields in \cite{Michel-Venkatesh-GL2,WuHan-GL2,Maga-Sub}, \nocite{Maga-Shifted} we shall obtain the first nontrivial estimate towards the metaplectic Ramanujan conjecture for $\widetilde{\SL}_2$ over arbitrary number fields. Recall that this is a key step in the settlement of Hilbert's eleventh problem in \cite{CPSS-Hilbert}, but the ground field therein is only needed to be totally real thanks to the work of Siegel.

%This article should be regarded as a mere addendum to the work of Baruch and Mao \cite{BaruchMao-Global}. As alluded to above, our (only) input is the local Bessel identities over $\BC$; see Remark \ref{rmk:Local Bessel Identity} for more discussions.

%\marginpar{\footnotesize Keep Introduction informal with minimal notations.}

\subsection{The formula of Waldspurger over arbitrary number fields}

To derive our Waldspurger formula, we shall closely follow the approach of Baruch and Mao in \cite{BaruchMao-Global}.
It is a combination of two results, the basic
Waldspurger's formula and Waldspurger's dichotomy result on theta correspondence.

Let $F$ be a number field and $\BA$ be its ring of ad\`eles. Given a nontrivial additive character $\psi$ on $\BA / F$,
we consider the global Shimura-Waldspurger correspondence $\Theta (\shskip\cdot\, , \psi)$ between automorphic representations $\pi$ of $\PGL_2 (\BA)$ and $\widetilde \pi$ of $\widetilde{\SL}_2 (\BA)$ \cite{Waldspurger-Shimura}.

The basic Waldspurger's formula is the following simple statement.

\begin{thm}[Restatement of Theorem \ref{thm:Waldspurger-Basic}]\label{thm: main1}
	Given $D \in F^{\times}$, define  $\psi^D(x) = \psi (D x)$, and let $\pi$ and $\widetilde{\pi}  = \Theta (\pi, \psi^D)$ correspond under the Shimura-Waldspurger correspondence with respect to $\psi^D$. We have
	\begin{align}\label{0eq: basis Waldspurger}
	\left|d_\pi(S,\psi)\right|^2L^S(\pi, 1/2)= |d_{\shskip\widetilde \pi } (S,\psi^D)  |^2.
	\end{align}
\end{thm}

In this theorem, $L^S(\pi,1/2)$ is the  central (partial) $L$-value for $\pi$, the constants $d_\pi(S,\psi)$ and $d_{\shskip\widetilde \pi } (S,\psi^D)$  (see \S \ref{sec: Fourier coefficients, d} for the definition) may be considered as the ``leading" and the $D$-th {\it Fourier coefficient} of $\pi$ and $\widetilde \pi $, respectively, where $S$ is a finite set of ``bad" local places.

Simply speaking, the identity \eqref{0eq: basis Waldspurger} follows from the comparison of two global distributions $I_{\pi}$ and $J_{\widetilde{\pi}}$ (the relative trace formula of Jacquet) and the relations between the attached local distributions $I_{\pi_{\varv}}$ and $J_{\widetilde{\pi}_{\varv} }$ (the local Bessel identities). See \S \ref{sec: Jacquet-RTF} and \ref{sec: local Bessel identity} for the details.

The Waldspurger formula for twisted central $L$-values is derived from the basic formula \eqref{0eq: basis Waldspurger} simply by replacing $\pi$ by $\pi \otimes \chiup_D$, where $\chiup_D$ is the quadratic character associated to $D$. This leads to the consideration of $\Theta (\pi \otimes \chiup_D, \psi^D)$ for varying $D \in F^{\times}$. Now the dichotomy result of Waldspurger in \cite{Waldspurger-Shimura3} gives an explicit classification of $\Theta (\pi \otimes \chiup_D, \psi^D)$ according to a certain finite partition of $F^{\times}$. The identity \eqref{0eq: basis Waldspurger} then yields a formula for $L^S(\pi \otimes \chiup_D, 1/2)$ in terms of the $D$-th Fourier coefficient $d_{\shskip\widetilde \pi } (S,\psi^D)$ of a fixed $\widetilde{\pi}$, as long as $D$ lies in the subset of $F^{\times}$ in the partition attached to $\widetilde{\pi}$. See \S \ref{sec: Waldspurger dichotomy} for the Waldspurger dichotomy and \S \ref{sec: Waldspurger, twisted} for the precise statement of the Waldspurger formula for $L^S(\pi \otimes \chiup_D, 1/2)$.

Some remarks on the recent works of Lapid-Mao and Qiu are in order. Lapid and Mao have a sequence of papers \cite{LM-1}--\cite{LM-4}\nocite{LM-2,LM-3,LM-Conj},
culminating in a Whittaker period formula for a cuspidal automorphic
representation $\widetilde{\pi}$ of $\widetilde{\mathrm{Sp}}_{2n}$ which is valid if $F$ is totally real and $\widetilde{\pi}_{\infty}$ is a discrete series. In \cite{Qiu-Whittaker}, Qiu proved this formula for $ \widetilde{\SL}_2 $ over an arbitrary number field. Though in a different form, this may also be considered as a generalization of Waldspurger's formula (the central value of $L$-functions for $\PGL_2$ is now related to the Whittaker periods of automorphic forms on $\widetilde{\SL}_2$) to all number fields. His approach is based on theta correspondence, and the most technical part is verifying the local identities in  \cite[Lemma 4.2]{Qiu-Whittaker} by the isometry property of quadratic Fourier transform and some asymptotic estimates for the matrix coefficients and  Whittaker functions of $\widetilde{\pi}_{\varv}$. Finally, we remark that,  if our basic Waldspurger's formula in Theorem \ref{thm:Waldspurger-Basic} is assumed, Qiu's Whittaker period formula for $\widetilde{\SL}_2$ may be deduced from that for $\PGL_2$, while the latter is already known (see \cite[\S 4]{LM-Conj} or \cite[\S 2]{Qiu-Whittaker}).

%In Baruch and Mao's project, the proof of the local Bessel identities is the technical part, while the proof of their Waldspurger formula is quite direct conceptually.

%\subsection{The Waldspurger formula for the central value of twisted $L$-functions}
	
\subsection{The metaplectic Ramanujan conjecture}

The first breakthrough in obtaining nontrivial bounds toward the Ramanujan conjecture for holomorphic modular forms of half-integral weight was achieved by Iwaniec in \cite{Iwaniec-Z/2}. Later, Duke \cite{Duke-Maass} got a similar bound for Maass forms.
In their work, they apply a Petersson or Kuznetsov trace formula and then proceed to bound sums of Sali\'e  sums. This approach is conceptually direct, and it never goes through the Waldspurger   formula.

	Alternatively, the Ramanujan-Lindel\"of equivalence in (1) established on the Waldspurger formula suggests that one may obtain nontrivial bounds toward the Ramanujan conjecture for cusp forms for $\widetilde{\SL}_2$ from subconvexity bounds for twisted $L$-functions for cusp forms of $\PGL_2$.
There is by now a great deal of machinery to deal with subconvexity problems, and so one can get even better bounds this way. Indeed, the bounds of Iwaniec and Duke  were improved in \cite{BHM-Mao,PM-Cubic} and \cite{Baruch-Mao-KZ-Maass} for holomorphic modular forms and Maass forms  of half-integral weight respectively\footnote{The bounds in \cite{BHM-Mao} and \cite{Baruch-Mao-KZ-Maass} may be further improved by the Burgess-type subconvexity bound in \cite{BH-Hybrid}. The results in \cite{BH-Hybrid} are valid for all Dirichlet characters. However, in our settings, we only need quadratic characters, and in this case the Weyl-type subconvexity bound was known earlier in the work of Conrey and Iwaniec \cite{CI-Cubic}.  A Weyl-type bound for the holomorphic modular forms and Maass forms in the Kohnen space for $\Gamma_0(4)$ may be deduced from  \cite{CI-Cubic} (a small issue is that the conductor of the character is assumed to be odd in \cite{CI-Cubic}).  In the same way, this is generalized by \cite{PM-Cubic} in the holomorphic case for arbitrary $\Gamma_0 (4 r)$ with $r$ odd and square-free. In view of the recent work \cite{Young-Kuznetsov}, the generalization in the Maass case should come in the near future.}. %It may be further improved by the Burgess and Weyl-type subconvexity bounds in \cite{BH-Hybrid}, \cite{CI-Cubic} and \cite{PM-Cubic}%
See also the aforementioned \cite{CPSS-Hilbert} and \cite{Blomer-Harcos-TR} in the case of totally real fields.

The $\GL_2$-subconvexity problem is now completely solved by Michel and Venkatesh \cite{Michel-Venkatesh-GL2} over arbitrary number fields. Later Han Wu \cite{WuHan-GL2} and Maga \cite{Maga-Shifted,Maga-Sub} proved by different methods the following subconvexity bound
\begin{equation}\label{1eq: subconvexity}
L (\pi \otimes \chiup, 1/2 ) \lll C(\chiup)^{ \beta }
\end{equation}
for all $\beta > 2 \widetilde{\theta} = \text{\Large$\frac 3 {8} \hskip -1 pt $} + \text{\Large$ \hskip -1 pt \frac {1} {4}$}\theta$, where $\pi$ is an automorphic representation of $\GL_2 (\BA)$, $\chiup$ is a Hecke character of $\BA^{\times}$ of (analytic) conductor $ C(\chiup)$. The $\theta$ is any exponent toward the Ramanujan-Petersson conjecture for $\GL_2 (\BA)$.  The best
known value $\theta = \text{\Large$ \hskip -1 pt \frac {7} {64}$}$ is due to Kim-Sarnak \cite{Kim-Sarnak} over $\BQ$, and to Blomer-Brumley \cite{Blomer-Brumley} over an arbitrary number field. Note that if $\theta = 0$ then we have  the Burgess exponent $ 2 \widetilde{\theta} = \text{\Large$\frac 3 {8} \hskip -1 pt $}$.

Now in our settings \eqref{1eq: subconvexity} would imply
\begin{align}\label{1eq: subconvexity, D}
L (\pi \otimes \chiup_D, 1/2 ) \lll |D|_{S_\infty}^{ \beta }, \hskip 15 pt    |D|_{S_\infty} \to  \infty,
\end{align}
if $D$ is a square-free integer in $F^{\times}$ (see \S \ref{sec: Waldspurger, twisted}). Here $S_\infty$ is the set of Archimedean places of $F$ and $|D|_{S_\infty}=\prod_{\varv \shskip \in S_\infty}|D|_\varv$. We shall obtain the following nontrivial bound  on the Fourier coefficients of a cusp form for  $\widetilde{\SL}_2 (\BA)$. See \S \ref{sec: Ramanujan-Lindelof} for the notations.

\begin{thm}
	 Let $\widetilde \pi $ be an irreducible cuspidal automorphic representation of $\widetilde{\SL}_2 (\BA)$ in $\widetilde A_{0\shskip 0}$. Let $\widetilde{\varphi}$ be a cusp form in the space of $ {\widetilde{\pi} }$. Let $D$ be a square-free integer in $F^{\times}$. Then, for any $\alpha > \widetilde{\theta} = \text{\Large$\frac 3 {16} \hskip -1 pt $} + \text{\Large$ \hskip -1 pt \frac {1} {8}$}\theta$, we have
	 \begin{align}\label{1eq: ineq Ramanujan}
	 \big|d_{\widetilde \pi}(\widetilde \varphi, S_\infty, \psi^D) \big|\lll |D|_{S_\infty}^{\alpha- \frac 1 2}
	 \end{align}
	 as $|D|_{S_\infty}\to \infty$, where the implied constant depends only on $\widetilde \pi, \widetilde \varphi$ and $\alpha$. The $\theta$ is any exponent toward the Ramanujan-Petersson conjecture for $\GL_2 (\BA)$.
\end{thm}

%We would like to emphasize that the metaplectic Ramanujan conjecture is not a spectral statement, in that it is not a consequence of the underlying local representations being tempered.

\subsection{Structure of the paper}

An outline of the paper is as follows. In \S \ref{sec: Fourier coefficients, d} we introduce the two constants  $d_{\pi} (S, \psi)$ and $d_{\widetilde\pi} (S, \psi)$. In \S \ref{sec: Waldspurger dichotomy} we recollect Waldspurger's results on theta correspondence \cite{Waldspurger-Shimura,Waldspurger-Shimura3}. In \S \ref{sec: Waldspurger formula} we give the explicit statement of our Waldspurger formula. In \S \ref{sec: Jacquet-RTF} we review the relative trace formula of Jacquet in \cite{Jacquet-RTF}. In \S \ref{sec: local Bessel identity} we review the local Bessel identities in \cite{BaruchMao-NA,BaruchMao-Real,Chai-Qi-Bessel}. In \S \ref{sec: proofs} we prove the  Waldspurger formula. In \S  \ref{sec: Ramanujan-Lindelof} we prove the  the Ramanujan-Lindel\"of equivalence and establish a nontrivial bound toward the metaplectic Ramanujan conjecture.
	
	\subsection*{Notation} %\marginpar{\footnotesize Try to keep the notation minimal. Omit the proof of some auxiliary lemmas. We probably do not need $\varpi_{\varv}$.}
	Let $F$ be a number field, and $\BA$ be its ad\`ele ring. For a place $\varv$ of $F$, let $F_{\varv}$ be the corresponding local field, and $|\ |_{\varv}$ denote the normalized metric on $F_{\varv}$. When $\varv$ is a non-Archimedean place, let $O_{\varv}$ be the ring of integers in $F_{\varv}$, %$P_{\varv}$ be its prime ideal, $\varpi_{\varv}$ be its uniformizer,
	and $q_{\varv}$ be the order of the residue field.
	
	For  $D \in F^{\times}$, let $\chiup_D $ be the quadratic character of $\BA^{\times} / F^{\times}$ associated to the field extension $F (\sqrt {D})$. Similarly, at a local place $\varv$, for $D \in F_{\varv}^{\times}$,
	we let $\chiup_D $  be the quadratic character of $F_{\varv}^{\times}$ associated to the extension $F_{\varv} (\sqrt {D})$.
	
	Let $\psi$ be a nontrivial additive character of $\BA / F$, with $\psi = \otimes_{\shskip\varv} \, \psi_{\varv}$. For $D \in F^{\times}$, define $\psi^D (x) = \psi (D x)$.
	
%	Sometimes, the $\varv$ will be suppressed in the notations when working on a local field.
	
\subsubsection*{Groups}	

Let $G = \GL_2$, $S =   {\SL}_2$ and $\widetilde S = \widetilde {\SL}_2$. %\marginpar{ \footnotesize $\widetilde{G}$ is not used later in the main text, nor in Jacquet.} %Let $1$ stand for the identity element of the group $G$, $\widetilde{G}$ or $\widetilde{S}$.
	Let $Z$ be the center of $G$, $B$ be the subgroup of $G$ consisting of upper triangular matrices,  $\widetilde{B}_S$ be the lift of $B \cap S$ in $\widetilde{S}$. Note that $G / Z = \PGL_2$. We shall use $1$ to denote the identity elements of these groups.
	
	Let us now briefly recall the definition of the metaplectic group $\widetilde S$. It consists of all pairs $(g, \epsilon)$,  $g \in S$, $\epsilon \in \{\pm 1 \}$, with the  law of multiplication
	\begin{align*}
	(g_1, \epsilon_1) (g_2, \epsilon_2) = (g_1 g_2, \alpha (g_1, g_2) \epsilon_1 \epsilon_2 ),
	\end{align*}
	where  $\alpha (g_1, g_2)$ is Kubota's cocycle that we now describe. Set
	\begin{equation*}
	x (g) = \left\{ \begin{split}
	& c \hskip 10 pt \text{ if } c \neq 0,\\
	& d \hskip 10 pt \text{ if } c = 0,
	\end{split}\right. \hskip 15 pt g = \begin{pmatrix}
	a & b \\ c & d
	\end{pmatrix}.
	\end{equation*}
	On the other hand the Hilbert symbol $(a, b)$ is defined by $(a, b) = 1$ if $a$ has the form $a = x^2 - b y^2$, $(a, b) = -1$ if not.
	Then
	\begin{align*}
	\alpha (g_1, g_2) = %\bigg( \det (g_1), \frac{x (g_1 g_2)}{x (g_1)} \bigg)
	\bigg(\frac{x (g_1 g_2)}{x (g_1)}, \frac{x (g_1 g_2)}{x (g_2)}\bigg).
	\end{align*}
The group  $\widetilde{S}$ thus fits into the exact sequence
\begin{align*}
1 \longrightarrow \{ \pm 1 \} \longrightarrow \widetilde{S} \longrightarrow S \longrightarrow 1.
\end{align*}
	Let
	\begin{align*}
& n(x) =  \begin{pmatrix}
1 & x \\ & 1
\end{pmatrix},  \hskip 10 pt  \widetilde{n} (x) = (n(x), 1), \hskip 10 pt  \varw = \begin{pmatrix}
 & 1 \\ 1 &
\end{pmatrix}, \hskip 10 pt \widetilde{\varw} = (\begin{pmatrix}
& -1 \\ 1 &
\end{pmatrix} \hskip -2 pt, 1), \\
& \hskip 25 pt  t (a) = \begin{pmatrix}
a & \\ & 1
\end{pmatrix}, \hskip 10 pt
\widetilde{s} (a) = (\begin{pmatrix}
a & \\ & a\-
\end{pmatrix} \hskip -2 pt, 1) , \hskip 10 pt z(c) = \begin{pmatrix}
c & \\ & c
\end{pmatrix}.
	\end{align*}
	
\subsubsection*{Measures}

We fix the additive measure $d x$ on $F_{\varv}$ to be self dual for the character $\psi_{\varv}$. Let the multiplicative measure be $\dx  a = (1-q_{\varv}\-)\- \cdot |a|_{\varv}\- d a$, where $q_{\varv}$ is the order of the residue field when $\varv$ is
non-Archimedean and $q_{\varv} = \infty$ when $\varv$ is Archimedean. Let $d g = |a|_{\varv} \shskip \dx c \shskip \dx a \shskip dx \shskip d y$ be the measure on $ G (F_{\varv})$ if we use the Bruhat coordinates $g = z(c)n(x)\varw \shskip t(a) n(y)$ on $G (F_{\varv}) \smallsetminus B (F_{\varv})$. The measure on $Z (F_{\varv})$
is $dz(c) = \dx c$, and we use the resulting quotient measure on $G(F_{\varv}) /Z(F_{\varv})$. For $g \in \widetilde{S} (\Fv) \smallsetminus \widetilde{B}_S (\Fv) $ with $g = \widetilde{n} (x) \widetilde{\varw} \shskip \widetilde{s} (a) \widetilde{n} (y)$, we define $d g = |a|_{\varv}^2 \shskip \dx a \shskip d x \shskip d y$ to be the measure on $\widetilde{S} (F_{\varv})$.

It should be stressed that the choice of additive measure does not matter for the statement of our theorems.

\subsubsection*{The Weil constant}

Define the Weil constant $\gamma (a, \psi_{\varv}^D)$ over $F_{\varv}$ to satisfy
\begin{align*}
\int \widehat{\Phi}^{2 D} (x) \psi_{\varv}^D (a x^2) d x = |a|_{\varv}^{-  1/ 2} \gamma (a, \psi^D)  \int  {\Phi} (x) \psi_{\varv}^D (- a\- x^2) d x,
\end{align*}
where $\Phi$ is a Schwartz-Bruhat function on $F_{\varv}$ and %$\widehat{\Phi}^D$ is its Fourier transform with respect to $\psi_{\varv}^D$,
\begin{align*}
\widehat{\Phi}^D (x) = \int \Phi (y) \psi_{\varv}^{D} (xy) d y.
\end{align*}
%We also set
%\begin{align*}
%\chiup (a, \psi_{\varv}^D) = \frac {\gamma (a, \psi_{\varv}^D)} {\gamma (1, \psi_{\varv}^D)} (-1, a).
%\end{align*}

\subsubsection*{Representations}

We use $\pi$ to denote an irreducible cuspidal representation of $G(\BA)$ with
trivial central character, and use $\widetilde {\pi}$ to denote an irreducible cuspidal representation
of $\widetilde{S}(\BA)$.
We have $\pi = \otimes_{\shskip\varv} \, \pi_{\varv}$ and $\widetilde\pi = \otimes_{\shskip\varv} \, \widetilde \pi_{\varv}$ as the restricted tensor products of representations
over local fields. Let $V_{\pi}$, $V_{\widetilde{\pi} }$, $V_{\pi_\varv}$, $V_{\widetilde{\pi}_\varv }$ denote the underlying
spaces of $\pi$, $\widetilde \pi$, $  \pi_{\varv}$ and $\widetilde \pi_{\varv}$ respectively. We  use $\|\varphi\|$ to denote the norm of a vector $\varphi$ in $V$; namely, let $\|\varphi\| = (\varphi, \varphi)^{ \frac 1  2}$ if $(\cdot, \cdot)$ is the Hermitian form on a space $V$.

%For a character $\mu$ of $F^{\times}_{\varv}$, let $\pi (\mu) = \pi (\mu, \mu\-)$ denote the principal series of $G (F_{\varv})$ induced from $\mu$ (\cite{J-L}) and $\widetilde{\pi} (\mu, \psi_{\varv})$ the    principal series of $\widetilde{S} (F_{\varv})$ induced from $ (s(a), \epsilon) \ra  \chiup (a, \psi_{\varv} ) \shskip \mu (a) \epsilon$ (see \cite[\S 2.3]{Jacquet-RTF} and \cite[\S I.4]{Waldspurger-Shimura3}).

Let the $L$-function $L(\pi, s)$ and the $\epsilon$-factor $\epsilon (\pi, s)$ be defined as in \cite{J-L}. We have $L(\pi, s) = \prod L(\pi_{\varv}, s)$ and $\epsilon (\pi, s) = \prod \epsilon (\pi_{\varv}, s, \psi_{\varv} )$; however $\epsilon (\pi_{\varv}, 1/2 ) = \epsilon (\pi_{\varv}, 1/2, \psi_{\varv} )$ is independent on $\psi_{\varv}$.

We use $S$ to denote a finite set of local places. As usual, let $S_{\infty}$ denote the set of all Archimedean places. Let $|\ |_{S}$ be the product of the local metrics over $S$. Let  $L^S (\pi, s) = \prod_{\varv \shskip \notin S} L (\pi_{\varv}, s)$ be the partial $L$-function above the complement of $S$. %For $\varphi = \otimes \, \varphi_{\varv}$ in $V_{\pi}$, we define in the same manner $\varphi^S = \otimes_{\varv \shskip \notin S} \varphi_{\varv}$.

We say that $S$
contains all bad places if it contains all $\varv$ that are Archimedean or have even residue characteristic. Then it is known that the covering  $\widetilde{\SL}_2 (F_{\varv})$ splits over $\SL_2 (O_{\varv})$ if $\varv \notin S$. Let $V_{\pi}^S$ or $V_{\tpi}^S$ be the subspace of vectors in $V_{\pi}$ or $V_{\tpi}$ which are invariant under $\prod_{\varv \shskip \notin S} \GL_2 (O_{\varv})$ or $\prod_{\varv \shskip \notin S} \SL_2 (O_{\varv})$ respectively.

%For $\varphi = \otimes \varphi_{\varv}$, let $ \varphi^S = $
	
\section{Definition of Two Constants}\label{sec: Fourier coefficients, d}

In this section, we define the two constants $d_{\pi} (S, \psi)$ and $d_{\widetilde\pi} (S, \psi)$ that occur in Theorem \ref{thm: main1}. For $D \in F^{\times}$, $d_{\pi} (S, \psi^D)$ and $d_{\widetilde\pi} (S, \psi^D)$ may be regarded as the {\it $D$-th Fourier coefficients} for $\pi$ and $\widetilde{\pi}$ respectively. %These constants are independent on the choice of additive measure.

\subsection{Definition of $d_{\pi} (S, \psi)$}\label{sec: defn of d} \

\vskip 5 pt

\subsubsection{Whittaker model on $G$}\label{sec: Whittaker model, G} Let $\pi$ be an irreducible cuspidal
automorphic representation of $G (\BA)$ with trivial central character. Let
$V_{\pi} \subset L^2 (Z(\BA) G(F) \backslash G(\BA) )$ be the underlying space of $\pi$. The global $\psi$-Whittaker model of $\pi$ consists of all the Whittaker functions $W_{\varphi}^{\psi}$ defined by
\begin{align}
W_{\varphi} (g) = W_{\varphi}^{\psi} (g) = \int_{\BA/F} \varphi (n(x) g) \psi (-x) d x, \hskip 10 pt \varphi \in V_{\pi}.
\end{align}

For any given local component $\pi_{\varv}$, we shall fix a nontrivial $\psi_{\varv}$-Whittaker functional $L_{\varv}$, satisfying
\begin{align*}
L_{\varv} (\pi_{\varv} (n(x)) \tv ) = \psi_{\varv} (x) L_{\varv} (\tv), \hskip 10 pt \tv \in V_{\pi_\varv}.
\end{align*}
	
Let $S$ be a finite set of places  that contains all bad places along with places $\varv$ where $\pi_{\varv}$ is not unramified. For $\varv \notin S$, let $\varphi_{\ro, \shskip \varv}$ be the unique vector in $V_{\pi_\varv}$ fixed under the action of $G(O_{\varv})$ such that $L_{\varv} (\varphi_{\ro, \shskip \varv}) = 1$. %Denote $\varphi_{\ro}^S = \otimes_{\varv \shskip \notin S} \shskip \varphi_{\ro, \shskip \varv}$.
Subsequently, we shall always assume that a pure tensor $\varphi = \otimes_{\shskip\varv} \, \varphi_{\varv}$ in $V_{\pi}^S$ has local component $\varphi_{ \varv} = \varphi_{\ro, \shskip \varv}$ outside $S$. %$\varphi_{\varv} = \otimes_{\varv \shskip \in S} \shskip \varphi_{\varv}   \text{\,\small $\otimes$} \, \varphi_{\ro}^S$.

We note that
\begin{align*}
L (\varphi) = W_{\varphi} (1)
\end{align*}
gives rise to a $\psi$-Whittaker functional on $V_{\pi}$. From the uniqueness of the local
Whittaker functional, $L$ can been expressed as a product of  $L_{\varv}$. Precisely, there is a constant $c_1 (\pi, S, \psi, \left\{L_{\varv}\right\})$ such that
\begin{equation}\label{2eq: c1}
W_{\varphi} (1) = c_1 (\pi, S, \psi, \left\{L_{\varv}\right\}) \prod_{ \varv \shskip \in S} L_{\varv} (\varphi_{\varv})
\end{equation}
for any pure tensor $\varphi \in V_{\pi}^S$. % with $\varphi^S = \varphi_{\ro}^S $ (as usual, $ \varphi^S $ is defined to be $\otimes_{\varv \shskip \notin S} \, \varphi_{ \varv}$).

\vskip 5 pt

\subsubsection{Hermitian forms for $G$} The space $V_{\pi}$ has the Hermitian form
\begin{align}
(\varphi, \varphi') = \int_{Z(\BA) G(F) \backslash G(\BA)} \varphi (g) \overline {\varphi' (g) } d g.
\end{align}
Given the nontrivial Whittaker functional $L_{\varv}$ on $\pi_{\varv}$, we can define a $G(F_{\varv})$-invariant Hermitian form on $V_{\pi_\varv}$ by
\begin{align}\label{2eq: Hermitian, local, G}
(\tv, \tv') = \int_{ F_{\varv}^{\times} } L_{\varv} (\pi_{\varv} (t (a)) \tv ) \overline{L_{\varv} (\pi_{\varv} (t (a)) \tv' )} \frac {d a} {|a|_{\varv}}, \hskip 10 pt \tv, \tv' \in V_{\pi_\varv}.
\end{align}
From the uniqueness of $G_{\varv}$-invariant Hermitian forms, we infer that there is a constant  $ c_2 (\pi, S, \psi, \left\{L_{\varv}\right\}) > 0$  such that
\begin{align}\label{2eq: c2}
\|\varphi \| = c_2 (\pi, S, \psi, \left\{L_{\varv}\right\}) \prod_{ \varv \shskip \in S} \|\varphi_{ \varv}\|,
\end{align}
for any  pure tensor $\varphi \in V_{\pi}^S$.

\vskip 5 pt

\subsubsection{The constant $d_{\pi} (S, \psi)$} We define
\begin{align}\label{2eq: defn of d}
d_{\pi} (S, \psi) = \left| c_1 (\pi, S, \psi, \left\{L_{\varv}\right\}) / c_2 (\pi, S, \psi, \left\{L_{\varv}\right\}) \right|.
\end{align}
It is easy to verify that $ d_{\pi} (S, \psi) $ is well defined in the sense that it is independent on the choice of the linear forms $L_{\varv}$. Also $ d_{\pi} (S, \psi) $ is independent on the the choice of the additive measure on $F_{\varv}$.
%This is the Fourier coefficient that we associate to $\pi$.
It follows from \eqref{2eq: c1} and \eqref{2eq: c2} that
\begin{align}\label{2eq: explicit d}
d_{\pi} (S, \psi) = \frac{|W_{\varphi}(1)|}{\|\varphi\|} \prod_{ \varv \shskip \in S} \frac {\|\varphi_{ \varv}\|} {|L_{\varv} (\varphi_{\varv}) | }
\end{align}
for  any pure tensor $\varphi \in V_{\pi}^S$ such that $L_{\varv} (\varphi_{\varv}) \neq 0$ for $\varv \in S$. We let $e (\varphi_{ \varv}, \psi_{\varv})$ denote the square of the local factor occurring in \eqref{2eq: explicit d}, namely,
\begin{align}\label{2eq: local e(phi)} e (\varphi_{ \varv}, \psi_{\varv}) =   {\|\varphi_{ \varv}\|^2} / { | L_{\varv} (\varphi_{ \varv}) |^2 }.
\end{align}
It is clear that the quotient on the right is independent on the choice of $L_{\varv}$.

\subsection{Definition of $d_{\shskip\widetilde{\pi}} (S, \psi)$} \

\vskip 5 pt

\subsubsection{Whittaker model on $\widetilde{S}$} We proceed in the same way as in \S \ref{sec: Whittaker model, G}. Let $\tpi$ be an irreducible cuspidal
automorphic representation of $\widetilde{S}$. For $\widetilde{\varphi} \in V_{\widetilde{\pi}} $, define
\begin{align}
	\widetilde W_{\widetilde\varphi} (g) = \widetilde W_{ \widetilde \varphi}^{\psi} (g) = \int_{\BA/F} \widetilde \varphi (\widetilde n(x) g) \psi (-x) d x.
\end{align}
Later, we shall use the abbreviation $\widetilde W_{\widetilde\varphi}^D (g) = \widetilde W_{ \widetilde \varphi}^{\psi^D} (g)$ for $D \in F^{\times}$.

Let us assume at the moment that $\tpi$ has a nontrivial $\psi$-Whittaker model. %namely, $\widetilde W_{\widetilde\varphi} (g) \neq 0$ for some  $\widetilde{\varphi} \in V_{\widetilde{\pi}} $.
Then locally each $\tpi_{\varv}$ has a nontrivial $\psi_{\varv}$-Whittaker
model, unique up to a scalar multiple. We shall  fix the corresponding $\psi_{\varv}$-Whittaker functional $\widetilde{L}_{\varv}$, satisfying
\begin{align*}
\widetilde L_{\varv} (\tpi_{\varv} (\widetilde n(x)) \widetilde{\tv} ) = \psi_{\varv} (x) \widetilde L_{\varv} (\widetilde{\tv}), \hskip 10 pt \widetilde{\tv} \in V_{\tpi_\varv}.
\end{align*}

Let $S$ be a finite set of places that  contains all bad places along with places $\varv$ where $\tpi_{\varv}$ is not unramified. For $\varv \notin S$, let $\widetilde\varphi_{\ro, \shskip \varv}$ be the unique vector in $V_{\tpi_\varv}$ fixed under the action of $\SL_2(O_{\varv})$ such that $\widetilde L_{\varv} (\widetilde \varphi_{\ro, \shskip \varv}) = 1$. %Denote $\widetilde \varphi_{\ro}^S = \otimes_{\varv \shskip \notin S} \shskip \widetilde \varphi_{\ro, \shskip \varv}$.
It will be understood that  a pure tensor $\widetilde \varphi = \otimes_{\shskip\varv} \, \widetilde \varphi_{\varv}$ in $V_{\tpi}^S$ has local component $\widetilde\varphi_{ \varv} = \widetilde\varphi_{\ro, \shskip \varv}$ outside $S$.

%From the uniqueness of Whittaker model, there exists a
Define the constant $\widetilde c_1 (\tpi, S, \psi, \{\widetilde L_{\varv}\})$ such that, for all  pure tensors $\widetilde \varphi $ in $V_{\tpi}^S$, %$\widetilde\varphi = \otimes_{\varv \shskip \in S} \, \widetilde\varphi_{\varv}   \text{\,\small $\otimes$} \, \widetilde \varphi_{\ro}^S$,
\begin{equation}\label{2eq: tilde c1}
\widetilde W_{\widetilde\varphi} (1) = \widetilde c_1 (\tpi, S, \psi,  \{\widetilde L_{\varv} \}) \prod_{ \varv \shskip \in S} \widetilde L_{\varv} (\widetilde \varphi_{\varv}) .
\end{equation}

\subsubsection{Hermitian forms for $\widetilde{S}$}   The space $V_{\tpi}$ has the Hermitian form
\begin{align}
(\widetilde\varphi, \widetilde\varphi') = \int_{S (F) \backslash \widetilde{S} (\BA)} \widetilde\varphi (g) \overline {\widetilde\varphi' (g) } d g;
\end{align}
here $S (F)$ splits in  $\widetilde{S} (\BA)$ so that we may regard $S (F)$ as a discrete subgroup of $\widetilde{S} (\BA)$.
The local Hermitian form on $ V_{\tpi_\varv} $ may be defined similar to \eqref{2eq: Hermitian, local, G}, though the definition is more complicated. According to \cite[\S 9]{BaruchMao-NA} and \cite[\S 15, 16]{BaruchMao-Real}, there is a choice of $\psi_{\varv}^{\delta}$-Whittaker functionals $\widetilde L_{\varv}^{\delta}$ (could be trivial)  on $ V_{\tpi_\varv} $ for $\delta$ in a collection of representatives of $F^{\times}_{\varv} / F^{\times \hskip 1 pt 2}_{\varv}$ containing $1$ ($\widetilde L_{\varv}^{1} = \widetilde L_{\varv}$), such that
\begin{align}
(\widetilde{\tv}, \widetilde{\tv}') = \sum_{\delta} \frac {\shskip|2|_{\varv}} 2 \int_{ F_{\varv}^{\times} } \widetilde L_{\varv}^{\delta} (\widetilde\pi_{\varv} (\widetilde{s} (a)) \widetilde{\tv} ) \hskip 1pt \overline{\widetilde L_{\varv}^{\delta} (\widetilde\pi_{\varv} (\widetilde{s} (a)) \widetilde{\tv}' )} \frac {d a} {|a|_{\varv}},  \hskip 10 pt \widetilde\tv, \widetilde\tv' \in V_{\tpi_\varv},
\end{align}
is an $\widetilde{S} (F_{\varv})$-invariant Hermitian form.
Define the constant $\widetilde c_2 (\tpi, S, \psi, \{\widetilde L_{\varv}\}) $ such that
\begin{align}\label{2eq: c2, S}
\|\widetilde\varphi \| = \widetilde c_2 (\tpi, S, \psi, \{\widetilde L_{\varv}\}) \prod_{ \varv \shskip \in S} \|\widetilde\varphi_{ \varv}\|
\end{align}
for all pure tensors $\widetilde \varphi $ in $V_{\tpi}^S$.

\subsubsection{The constant $d_{\shskip \tpi} (S, \psi)$} We define
\begin{align}
d_{\shskip \tpi} (S, \psi) =  \big|\widetilde c_1 (\tpi, S, \psi,  \{\widetilde L_{\varv} \}) / \widetilde c_2 (\tpi, S, \psi,  \{\widetilde L_{\varv} \})  \big|.
\end{align}
This definition is  independent on the choice of the linear forms $\widetilde L_{\varv}$ and the additive measure on $F_{\varv}$.
%This is the Fourier coefficient that we associate to $\tpi$.
Additionally, we set
$d_{\shskip \tpi} (S, \psi) = 0$ when $\tpi$ does not have a nontrivial $\psi$-Whittaker model.
We have
\begin{align}\label{2eq: explicit d tilde}
d_{\shskip \tpi} (S, \psi) = \frac{|\widetilde W_{\widetilde \varphi}(1)|}{\|\widetilde \varphi\|} \prod_{ \varv \shskip \in S} \frac {\|\widetilde \varphi_{ \varv}\|} {|\widetilde L_{\varv} (\widetilde \varphi_{\varv}) | }
\end{align}
for any pure tensor $\widetilde \varphi \in V_{\tpi}^S$ such that $\widetilde L_{\varv} (\widetilde \varphi_{\varv}) \neq 0$ for $\varv \in S$. We let $e (\widetilde\varphi_{ \varv}, \psi_{\varv})$ denote the square of the local factor in \eqref{2eq: explicit d tilde}, namely,
\begin{align}\label{2eq: local e(tilde phi)}
& e (\widetilde\varphi_{ \varv}, \psi_{\varv}) =  {\|\widetilde\varphi_{ \varv}\|^2} / { | \widetilde L_{\varv} (\widetilde\varphi_{ \varv}) |^2 }.
\end{align}
Note that the quotient on the right is independent on the choice of $\widetilde L_{\varv}$.
When working with  $\psi_{\varv}^D$, we shall denote by  $\widetilde{L}_{\varv}^D$ the $\psi_{\varv}^D$-Whittaker functional and by $\| \widetilde\varphi_{ \varv} \|_{D}$ the local Hermitian norm of $\widetilde \varphi_{ \varv}$ constructed from $ \widetilde{L}_{\varv}^D $.
	
	\section{Results on the Theta Correspondence \`a la Waldspurger}\label{sec: Waldspurger dichotomy}
	
	In this section, we describe Waldspurger's beautiful results  in \cite{Waldspurger-Shimura, Waldspurger-Shimura3} on the theta correspondence, both local and global, between $\PGL_2 $ and $\widetilde {\SL}_2 $.
	
	Given $\psi = \otimes_{\shskip\varv} \, \psi_{\varv}$, let $\theta (\shskip\cdot \shskip,\psi_\varv)$ and $\Theta (\shskip\cdot \shskip,\psi)$ be the local and global  theta correspondence, respectively, as defined in \cite{Waldspurger-Shimura, Waldspurger-Shimura3}.
	%We shall denote $\widetilde \pi^D=\Theta( \pi\otimes \chi_D, \psi^D)$ and $\widetilde \pi_\varv^D=\theta(\pi_\varv\otimes \chi_D, \psi_\varv^D)$.
	
	\subsection{Local theory of Waldspurger}
	
Let $\pi_\varv$ be an irreducible unitary representation of $\PGL_2(F_\varv)$. For $D\in F_\varv^\times$, define
	\begin{align*}
	\epsilon({D}, {\pi_\varv}) =\chiup_D(-1)\epsilon(\pi_\varv,1/2)/\epsilon(\pi_\varv\otimes \chiup_D, 1/2),
	\end{align*}
	which takes value in $\{ \pm 1 \}$.
	Set
	\begin{align*}
	F_\varv^{\pm}(\pi_\varv)=\left\{D\in F_\varv^\times   :   \epsilon({D}, {\pi_\varv})=\pm 1\right\},
	\end{align*}
	 and we have a partition $F_\varv^\times=F_\varv^+(\pi_\varv)\cup F_\varv^-(\pi_\varv). $

	%We summarize the local results of Waldspurger as follows.
	
	\begin{thm}[Waldspurger]\label{thm: Waldspurger-Local} Let $P_{0,\shskip \varv}$ be the set of equivalence classes of discrete series of $\PGL_2(F_\varv)$.
		
		{\rm(1).} When $\pi_\varv\notin P_{0,\shskip\varv}$, $F_\varv^\times(\pi_\varv)=F_\varv^\times$ and $\theta(\pi_\varv\otimes \chiup_D, \psi_\varv^D) =\theta(\pi_\varv, \psi_\varv)$ for all $D \in F_{\varv}^{\times}$.
		
		{\rm(2).} When $\pi_\varv\in P_{0,\shskip\varv}$, there are two distinct representations $\widetilde \pi_\varv^+=\theta(\pi_\varv,\psi_\varv)$ and $\widetilde \pi_\varv^-$ of $\widetilde {\SL}_2(F_\varv)$, such that
		\begin{equation*}
\theta(\pi_\varv\otimes \chiup_D, \psi_\varv^D)= \left\{ \begin{split}
		&\widetilde \pi_\varv^+ \hskip 10 pt \text{ if } D\in F_\varv^+(\pi_\varv), \\
		&\widetilde \pi_\varv^- \hskip 10 pt \text{ if } D\in F_\varv^-(\pi_\varv).
\end{split}\right.
		\end{equation*}

        {\rm(3).} $\theta(\pi_\varv\otimes \chiup_D, \psi_\varv^D)= \widetilde\pi_\varv^{\pm}$ if and only if $\widetilde\pi_\varv^{\pm}$ has a nontrivial $\psi_\varv^D$-Whittaker model.	
	\end{thm}

    \subsection{Global theory of Waldspurger}\label{sec: golbal Waldspurger}

Let $\widetilde A_{0 \shskip 0}$ be the set of irreducible cuspidal automorphic representations of $\widetilde {\SL}_2(\BA)$ that are orthogonal to the theta series attached to one dimensional quadratic forms. These are precisely those representations that admit the Shimura  lift. For  $\widetilde \pi_1$ and $ \widetilde \pi_2$ in $\widetilde A_{0\shskip 0}$, we say they are near equivalent, denoted by $\widetilde \pi_1 \thicksim \widetilde \pi_2$, if  $\widetilde \pi_{1,\shskip\varv}\cong \widetilde \pi_{2,\shskip\varv}$ for almost all places $\varv$. Use $\widebar{A}_{0\shskip 0}$ to denote the quotient of $\widetilde A_{0\shskip 0}$ by this relation.

Let $A_{0, \shskip i}$ be the set of irreducible cuspidal automorphic representations of $\PGL_2(\BA)$ %such that for any subrepresentation $\pi$ of $A_{{0}, \shskip {i}} $, there is $D\in F^\times$,
such that $L(\pi\otimes \chiup_D, 1/2)\neq 0$ for some  $D\in F^\times$.

Let $\varSigma=\varSigma(\pi)$ be the set of places $\varv$ such that $\pi_\varv\in P_{0,\shskip\varv}$. Given $D\in F^\times$, let $\epsilon(D,\pi)=(\epsilon(D,\pi_\varv))_{\varv \shskip\in \varSigma}$. Note that $\epsilon(D,\pi)\in \{\pm 1\}^{|\varSigma|}$. We have
\begin{align}
\epsilon(\pi\otimes \chiup_D,1/2)=\epsilon(\pi, 1/2)\prod_{\varv\shskip\in \varSigma}\epsilon(D,\pi_\varv).
\end{align}
We  shall use $ {\epsilon} =(\epsilon_\varv)_{\varv \shskip\in \varSigma}$ to denote an element in $ \{\pm 1\}^{|\varSigma|} $. %, with $\epsilon_{\varv} \in \{ \pm 1 \}$.
Given such an $\epsilon$, define
$$F^{\epsilon} (\pi) = \left\{ D \in F^{\times} : \epsilon (D, \pi) = \epsilon  \right \}.$$
Then we get a partition
$$ F^{\times} = \underset{\epsilon  \shskip\in \{\pm 1\}^{|\varSigma|}}{\text{$\bigcup$}}   F^{\epsilon} (\pi).  $$

\begin{thm}[Waldspurger]\label{thm:Waldspurger-Global} Let notations be as above.
	
{\rm(1).} $\Theta(\pi,\psi)\cong \otimes_{\shskip\varv} \shskip\theta(\pi_\varv,\psi_\varv)$ if $\Theta(\pi,\psi)\neq 0${\rm;} $\Theta(\widetilde \pi,\psi)\cong \otimes_{\shskip\varv} \shskip \theta(\widetilde \pi_\varv,\psi_\varv)$ if $\Theta(\widetilde \pi,\psi)\neq 0$.

{\rm(2).} $\Theta(\pi,\psi)\neq 0$ if and only if $L(\pi,1/2)\neq 0$. $\Theta(\widetilde \pi,\psi)\neq 0$ if and only if $\widetilde \pi$ has a nontrivial $\psi$-Whittaker model.

{\rm(3).} For $\widetilde \pi \in \widetilde A_{0\shskip 0}$, there is a unique $\pi$ associated to $\widetilde \pi$, such that $\Theta(\widetilde \pi,\psi^D)\otimes \chiup_D=\pi$ whenever $\Theta(\widetilde \pi, \psi^D)\neq 0$. Denote this association by $\pi=S_\psi(\widetilde \pi)$. The map $S_\psi$ then defines a bijection between $\widebar{A}_{0\shskip 0}$ and $A_{0,\shskip i}$.

{\rm(4).} Let $\pi=S_\psi(\widetilde \pi)$. There is a bijection between  the near equivalence class of $\widetilde \pi$ and the subset of $ \{\pm1\}^{|\varSigma|}$ comprising those $\epsilon $ such that $\prod_{\varv\shskip\in \varSigma}\epsilon_\varv= \epsilon(\pi, 1/2)${\rm;} denote by $\widetilde\pi^{\epsilon} $ the representation corresponding to $\epsilon$. For convenience, set $\widetilde \pi^\epsilon=0$ if $\prod_{\varv\shskip\in \varSigma}\epsilon_\varv\neq \epsilon(\pi,1/2)$.

{\rm(5).} If $\pi=S_\psi(\widetilde \pi)$, then the near equivalence class of $\widetilde \pi$ consists of all the nonzero $\Theta(\pi\otimes \chiup_D, \psi^D)$. Precisely, for $D\in F^\epsilon(\pi)$,
\begin{equation*}
\Theta(\pi\otimes \chiup_D, \psi^D)=
\left\{ \begin{split}
&\widetilde \pi^\epsilon, \hskip 10 pt \text{ if } L(\pi\otimes \chiup_D,1/2)\neq 0, \\
&0, \hskip 14 pt \text{ otherwise.}
\end{split}\right.
\end{equation*}

%Let $\epsilon\in \{\pm1\}^{|\varSigma|}$. If $\prod_{\varv\shskip\in \varSigma}\epsilon_\varv\neq \epsilon(\pi, 1/2)$, then $\widetilde \pi^D=0$ for all $D\in F^\epsilon(\pi)$. If $\prod_{\varv\shskip\in \varSigma}\epsilon_\varv=\epsilon(\pi, 1/2)$, then there is a unique $\widetilde \pi^\epsilon$ such that for $D\in F^\epsilon(\pi)$, $\widetilde \pi^D=\widetilde \pi^\epsilon$ when $L(\pi\otimes \chiup_D,1/2)\neq 0$ and $\widetilde \pi^D=0$ otherwise.

\end{thm}

	\section{The Formula of Waldspurger}\label{sec: Waldspurger formula}
	
	%We now have introduced all the notations for our main theorems.
	
	We are now ready to state our main theorems. Their proofs will be provided in \S \ref{sec: proofs} after the relative trace formula and local Bessel identities are introduced.
	%The formulae of Waldspurger-type in these theorems were established by Baruch and Mao \cite{BaruchMao-Global} over totally real fields. As indicated in Introduction, we are able to extend them over arbitrary number field thanks to the local Bessel identity over the complex field.
	
	\subsection{The basic Waldspurger's formula}
	
%	If $S$ is a finite set of places of $F$, recall that $L^S(\pi,s)$ is the partial $L$-function $\prod_{\varv\notin S}L(\pi_\varv, s)$.

\begin{thm}\label{thm:Waldspurger-Basic}
Let $\pi$ be an irreducible cuspidal automorphic representation of $\GL_2(\BA)$, trivial on the center, with $L(\pi, 1/2)\neq 0$. Let $D \in F^{\times} $. Set $\widetilde \pi_D=\Theta(\pi , \psi^D)$. Suppose that $S$ is a finite set of places of $F$ which contains all bad places, all places $\varv$ where $\psi_{\varv}$ or $\psi_{\varv}^D$ is ramified, along with all places where $\pi_\varv$ or $\widetilde \pi_{D, \shskip \varv}$ is not unramified. We have
 \begin{align}\label{4eq: basis Waldspurger}
 \left|d_\pi(S,\psi)\right|^2L^S(\pi, 1/2)= |d_{\shskip\widetilde \pi_D} (S,\psi^D)  |^2.
 \end{align}
Recall that $L^S(\pi,s)$ is the partial $L$-function $\prod_{\varv \shskip \notin S}L(\pi_\varv, s)$ and that $d_\pi(S,\psi)$ and $d_{\shskip\widetilde \pi_D} (S,\psi^D)$ are the two constants defined as in {\rm\S  \ref{sec: Fourier coefficients, d}}.
\end{thm}

In view of \eqref{2eq: explicit d} and \eqref{2eq: explicit d tilde}, we may reformulate \eqref{4eq: basis Waldspurger} more explicitly as follows.

\begin{cor}
	Let $\varphi \in V_{\pi}^S$ and $\widetilde\varphi \in V_{\widetilde \pi_D}^S$ be pure tensors such that $L_{\varv} (\varphi_{ \varv}) \neq 0$ and $\widetilde L_{\varv} (\widetilde\varphi_{ \varv}) \neq 0 $ for all $\varv \in S$.  Let the local constants $e ( \varphi_{ \varv}, \psi_{\varv})$ and $e (\widetilde\varphi_{ \varv}, \psi_{\varv}^D)$ be defined as in \eqref{2eq: local e(phi)} and \eqref{2eq: local e(tilde phi)} respectively. Then
\begin{align}
\frac{|\widetilde W_{\widetilde \varphi}^D(1)|^2}{\|\widetilde \varphi\|^2}=\frac{|W_\varphi(1)|^2L(\pi, 1/2)}{\|\varphi\|^2}\prod_{\varv\shskip\in S}\frac{e(\varphi_\varv, \psi_\varv)}{e(\widetilde \varphi_\varv, \psi_\varv^D)L(\pi_\varv, 1/2)}.
\end{align}
\end{cor}
	
\subsection{The Waldspurger formula for twisted central $L$-values}\label{sec: Waldspurger, twisted}

Applying Theorem \ref{thm:Waldspurger-Basic} to $\pi\otimes \chiup_D$, we get the Waldspurger formula for $L(\pi \otimes \chiup_D, 1/2)$.

For $D\in F^\times$, we say, with slight abuse of terminology, that $D$ is a {\it square-free} integer  if $|4|_{\varv}  q_{\varv} \- \leqslant |D|_\varv \leqslant 1$ for all non-Archimedean places $\varv$; clearly, one has $|D|_\varv=1$ or $ q_\varv^{-1}$ if the place $\varv$ is \textit{odd}.
%\begin{itemize}
%	\item [-] $|D|_\varv=1$ or $ q_\varv^{-1}$ for all \textit{odd} non-Archimedean places $\varv$.
%\item [-] $|4|_{\varv}  q_{\varv} \- \leqslant |D|_\varv \leqslant 1$ for all \textit{even} non-Archimedean places $\varv$.
%\end{itemize}
Note that the discriminant of any quadratic extension of $F$ is a square-free integer.

According to Theorem \ref{thm:Waldspurger-Global}, there is a bijection $S_{\psi}$ between $\widebar A_{0\shskip 0}$ and $A_{0,\shskip i}$. %Recall that $\widebar A_{0\shskip 0}$ is the quotient of $\widetilde A_{0\shskip 0}$ by near equivalence relation.
For $\pi\in A_{0,\shskip i}$ let $\big\{\widetilde \pi^\epsilon : \prod_{\varv\shskip\in \varSigma}\epsilon_\varv= \epsilon(\pi, 1/2) \big\}$ be the packet $S^{-1}_{\psi} (\pi)$ of representations in $\widetilde A_{0\shskip 0}$ corresponding to $\pi$. %which are mapped to $\pi$ under the above bijection.
Here $\epsilon = (\epsilon_\varv)_{\varv \shskip \in \varSigma} \in \{\pm 1 \}^{|\varSigma|}$, with $\varSigma=\varSigma(\pi)$ defined as in \S \ref{sec: golbal Waldspurger}. %Then $\pi=S_\psi(\widetilde \pi^\epsilon)$ if $\widetilde \pi^\epsilon\neq 0$.
For $D\in F^\epsilon(\pi)$, we have $\Theta(\pi\otimes \chiup_D, \psi^D)=\widetilde \pi^\epsilon$ or $0$ according as $L(\pi\otimes \chiup_D,1/2)\neq 0$ or not.

\begin{thm}\label{thm:Waldspurger-Twist}
Let $\pi$ and $\widetilde \pi^{\epsilon}$ be as above. Let $S$ be a finite set of places containing all bad places along with the places $\varv$ where $\pi_\varv$ or $\psi_\varv$ is not unramified. Then for any square-free integer $D\in F^\times$, if $D\in F^{\epsilon_{\oldstylenums{0}}}(\pi)$, then $d_{\shskip\widetilde \pi^\epsilon}(S, \psi^D)=0$ whenever $\epsilon\neq \epsilon_{\oldstylenums{0}}$, but, for this $\epsilon_{\oldstylenums{0}}$, we have
\begin{align}\label{4eq: Waldspurger twisted}
|d_{\shskip\widetilde \pi^{\shskip\epsilon_{\oldstylenums{0}}}}(S, \psi^D) |^2=|d_\pi(S, \psi)|^2 L^S(\pi\otimes \chiup_D, 1/2) / |D|_S. %\prod_{\varv\shskip\in S}|D|_\varv .
\end{align}
More explicitly, for pure tensors $\varphi=\otimes_{\varv} \, \varphi_\varv\in V_\pi^S$ and $\widetilde \varphi=\otimes_{\varv} \, \widetilde \varphi_\varv\in V_{\widetilde \pi^{\shskip\epsilon_{\oldstylenums{0}}}}^S$, with $\varphi_\varv=\varphi_{\ro,\shskip \varv}$ and $\widetilde \varphi_{\varv}=\widetilde \varphi_{\ro,\shskip \varv}$ for all $\varv\notin S$, we have
\begin{align}
\frac{ |\widetilde W_{\widetilde \varphi}^D(1) |^2}{\|\widetilde \varphi\|^2}=\frac{|W_\varphi(1)|^2L(\pi\otimes \chiup_D, 1/2)}{\|\varphi\|^2}\prod_{\varv\shskip\in S}\frac{e(\varphi_\varv, \psi_\varv)}{e(\widetilde \varphi_\varv, \psi_\varv^D)L(\pi_\varv\otimes \chiup_D, 1/2)|D|_\varv}.
\end{align}
\end{thm}

%\begin{align}
%\sum_{\epsilon\in \{\pm 1\}^{|\Sigma|}}|d_{\widetilde \pi^\epsilon}(S, \psi^D)|^2=|d_\pi(S, \psi)|^2 L^S(\pi\otimes \chi_D, 1/2) \prod_{\varv\in S}|D|_\varv^{-1}.
%\end{align}

%{\rm(2).} Only one of the terms in the above sum is nonzero. If $D\in F^{\epsilon_{\oldstylenums{0}}}(\pi)$, then $d_{\widetilde \pi^\epsilon}(S, \psi^D)=0$ whenever $\epsilon\neq \epsilon_{\oldstylenums{0}}$.

%{\rm(3).} More explicitly, if $D\in F^{\epsilon_{\oldstylenums{0}}}(\pi)$, then for $\widetilde \pi=\widetilde \pi^{\epsilon_{\oldstylenums{0}}}$, we have
%\begin{align}
%|d_{\widetilde \pi}(S, \psi^D)|^2=|d_\pi(S, \psi)|^2 L^S(\pi\otimes \chi_D, 1/2) \prod_{\varv\in S}|D|_\varv^{-1}.
%\end{align}
%For pure tensor $\varphi=\otimes \varphi_\varv\in V_\pi$ and $\widetilde \varphi=\otimes \widetilde \varphi_\varv\in V_{\widetilde \pi}$ with $\varphi_\varv=\varphi_{\ro, \varv}$ and $\widetilde \varphi_{\varv}=\widetilde \varphi_{\ro, \varv}$ if $\varv\notin S$,
%\begin{align}
%\frac{|\widetilde W_{\widetilde \varphi}^D(1)|^2}{\|\widetilde \varphi\|^2}=\frac{|W_\varphi(1)|^2L(\pi\otimes \chi_D, 1/2)}{\|\varphi\|^2}\prod_{\varv\shskip\in S}\frac{e(\varphi_\varv, \psi_\varv)}{e(\widetilde \varphi_\varv, \psi_\varv^D)L(\pi_\varv\otimes \chi_D, 1/2)|D|_\varv}.
%\end{align}

\section{Review of a Relative Trace Formula of Jacquet}\label{sec: Jacquet-RTF}

In this section, we give a brief review of the relative trace formula of Jacquet in \cite{Jacquet-RTF}. Note that he acknowledges the inspiration from Iwaniec's work \cite{Iwaniec-Z/2,Iwaniec-Wald}.

\subsection{Definition of the global distributions $I(f, \psi)$ and $J (f', \psi^D)$}

Let $C_c^{\infty} (Z(\BA) \backslash G(\BA))$ denote the space of smooth compactly supported functions on $Z(\BA) \backslash G(\BA)$. For $f \in C_c^{\infty} (Z(\BA) \backslash G(\BA))$ define a kernel function
\begin{align*}
	K (x, y; f) = \sum_{ \xi \shskip\in Z(F) \backslash G(F)} f (x\- \xi y).
\end{align*}
Define the distribution $I(f, \psi)$ to be
\begin{align}\label{5eq: defn I(f)}
	I(f, \psi) = \int_{\BA^{\times}/F^{\times} } \int_{\BA/F} K (t(a), n(x); f) \psi (x) d x \shskip \dx a.
\end{align}
Let $C_c^{\infty} (\widetilde{S} (\BA))$ denote the space of  {\it genuine} smooth compactly supported functions on $\widetilde{S} (\BA)$. For $f' \in C_c^{\infty} (\widetilde{S} (\BA))$ define
\begin{align*}
	K (x, y; f') = \sum_{ \xi \shskip\in S(F)  } f' (x\- \xi y).
\end{align*}
Define the distribution $J(f', \psi^D)$ to be
\begin{align}\label{5eq: defn J(f')}
	J(f', \psi^D) = \int_{\BA/F} \int_{\BA/F} K  (\widetilde n(x), \widetilde n(y); f') \psi^D (-x + y) d x \shskip d y.
\end{align}

The relative trace formula is an identity between the distributions $I(f; \psi)$ and $J (f'; \psi^D)$.

\subsection{Comparison of orbital integrals}
The relative trace formula identity follows from comparing the orbital integrals on the geometric side. See \cite[\S 3.1, 4, 6.1, 6.2, 7]{Jacquet-RTF}.

\begin{prop}[Jacquet]\label{prop: Jacquet 1} Suppose that $f $ and $f'$ are products
	of local functions $f_{\varv} \in C_c^{\infty} (Z(F_{\varv} ) \backslash G(F_{\varv} ) )$ and $f'_{\varv} \in C_c^{\infty} (\widetilde{S} (F_{\varv}))$, that is, $f = \otimes_{\varv} \, f_{\varv} $ and $f' = \otimes_{\varv} \, f_{\varv}' $. We have
	\begin{align*}
		I(f, \psi) &= \prod S_{\psi_{\varv}}^{\hskip -1 pt +} (f_{\varv}) + \prod S_{\psi_{\varv}}^{\hskip -1 pt -} (f_{\varv}) + \sum_{ \xi \shskip \in F^{\times} } \prod O_{\shskip\psi_{\varv}} (n(\xi) \varw; f_{\varv}), \\
		J(f', \psi^D) &= \prod S_{\psi^D_{\varv}}^{\hskip -1 pt +} (f_{\varv}') + \prod S_{\psi^D_{\varv}}^{\hskip -1 pt -} (f_{\varv}') + \sum_{ \xi \shskip \in F^{\times} } \prod O_{\shskip\psi^D_{\varv}} (  \widetilde \varw \widetilde s (\xi); f_{\varv}'),
	\end{align*}
	in which
	\begin{align*}
		O_{\shskip\psi_{\varv}} (g; f_{\varv}) & = \int_{F_{\varv}^{\times} } \int_{ F_{\varv} } f_{\varv} (t(a) \shskip g \shskip n (x) )\psi_{\varv} (x) dx \shskip \dx a, \\
		O_{\shskip\psi^D_{\varv}} (  g ; f_{\varv}') & = \int_{ F_{\varv} }   \int_{ F_{\varv} } f_{\varv}' (\widetilde n (x)\shskip g \shskip\widetilde n (y) ) \psi_{\varv}^D (x+y) dx \shskip dy,
	\end{align*}
	while $S_{\psi_{\varv}}^{\hskip -1 pt \pm} (f_{\varv})$ and $S_{\psi^D_{\varv}}^{\hskip -1 pt \pm} (f_{\varv}')$ are certain singular orbital integrals which are determined by the asymptotics of $O_{\shskip\psi_{\varv}} (n(a) \varw; f_{\varv})$ and  $O_{\shskip\psi^D_{\varv}} (  \widetilde \varw \widetilde s (a); f_{\varv}')$ at the boundary $a = 0$ respectively.
\end{prop}

\begin{prop}[Jacquet] \label{prop: Jacquet 2} %Let $F$ be a local field.
	For each $f_{\varv} \in C_c^{\infty} (Z(F_{\varv})\backslash G(F_{\varv}) )$ there exists $f'_{\varv} \in C_c^{\infty} (\widetilde{S} (F_{\varv}))$ such that for $a \in F_{\varv}^{\times}$
	\begin{equation*}
		O_{\shskip\psi_{\varv}} (n(a/4D) \varw; f_{\varv}) = O_{\shskip\psi^D_{\varv}} (  \widetilde \varw \widetilde s (a); f_{\varv}') \cdot \psi_{\varv} (-2D/a) |a|_{\varv}^{1/2} / \gamma (a, \psi_{\varv}^D) ,
	\end{equation*}
	and
	\begin{align*}
		S_{\psi_{\varv}}^{\hskip -1 pt \pm} (f_{\varv} ) = S_{\psi_{\varv}^D }^{\hskip -1 pt \pm} (f_{\varv} ') \cdot (\pm 1, -1)_{\varv} \gamma (\pm 1, \psi_{\varv}^D) / \gamma (1, \psi_{\varv}^D)  .
	\end{align*}
	Conversely, given $f_{\varv}' \in C_c^{\infty} (\widetilde{S} (F_{\varv}))$ we can find  $f_{\varv} \in C_c^{\infty} (Z(F_{\varv})\backslash G(F_{\varv}) )$ satisfying the above equations. We say that the two functions $f_{\varv}$ and $f'_{\varv}$ match if the relations are satisfied.
\end{prop}

Moreover, Jacquet established the following result on the matching between Hecke functions on $Z(F_{\varv} ) \backslash G(F_{\varv} )$ and $\widetilde{S} (F_{\varv})$ at almost all places (\cite[\S 2, 5, 8]{Jacquet-RTF}).

\begin{prop}[Jacquet] \label{prop: Jacquet 3}
	Suppose that $F_{\varv}$ is a local non-Archimedean field, of odd residual characteristic, and that $\psi_{\varv} $ and $\psi_{\varv}^D $ are both unramified. Then there is a canonical choice of  isomorphism  between the Hecke algebras
	\begin{equation}\label{5eq: Hecke isom}
		{H} (Z(F_{\varv} ) \backslash G(F_{\varv} ) )  \xrightarrow{\hskip 3 pt\sim \hskip 3 pt } {H} (\widetilde{S} (F_{\varv}))
	\end{equation} such that Hecke
	functions $f_{\varv} $ and $f_{\varv} '$ match if they correspond under the isomorphism \eqref{5eq: Hecke isom}.
\end{prop}

Combining Proposition \ref{prop: Jacquet 1}, \ref{prop: Jacquet 2} and \ref{prop: Jacquet 3}, we have the following theorem.

\begin{thm}[Jacquet]\label{thm: I=J}
	Fix a finite set of places $S$ that contains all the bad places and the places where $\psi$ or $\psi^D$ is ramified. Let $f $ and $f'$ be as in Proposition {\rm\ref{prop: Jacquet 1}}. Assume that the local functions $f_{\varv}$ and $f_{\varv}'$ match as in Proposition {\rm\ref{prop: Jacquet 2}}  for each $\varv$ and that $f_{\varv} \in {H} (Z(F_{\varv} ) \backslash G(F_{\varv} ) ) $ and $ f_{\varv} \in {H} (\widetilde{S} (F_{\varv})) $ correspond to one another via the isomorphism {\rm\eqref{5eq: Hecke isom}} in  Proposition {\rm\ref{prop: Jacquet 3}} for each $\varv \notin S$. Then
	\begin{align*}
		I (f, \psi) = J (f', \psi^D).
	\end{align*}
\end{thm}

\subsection{Connection to the Shimura-Waldspurger correspondence}

For an irreducible cuspidal  representation  $\pi$ of $G(\BA)$, trivial on the central, define
\begin{align}\label{5eq: defn of I pi}
	I_{\pi} (f, \psi) = \sum_{\varphi_{\mathfrak i}} Z (\pi (f) \varphi_{\mathfrak i}) \overline {W_{\varphi_{\mathfrak i}} (1)},
\end{align}
with $\{\varphi_{\mathfrak i}\}$  an orthonormal basis  of $V_{\pi}$; here for $\varphi \in V_{\pi}$
\begin{align*}
	\pi (f) \varphi & = \int_{Z(\BA) \backslash G(\BA)} f(g) \pi (g) \varphi \, d g, \\
	W_{\varphi}  (1)  = \int_{\BA/F}   \varphi  (  n(&x))  \psi (-x) d x,  \hskip 10 pt Z (\varphi )  = \int_{\BA^{\times} \hskip -1 pt / F^{\times} } \varphi (t(a)) \dx a.
\end{align*}
For  an irreducible cuspidal  representation $\widetilde{\pi}$ of $\widetilde{S} (\BA)$ define
\begin{align*}\label{5eq: defn of J pi}
	J_{\widetilde{\pi}} (f', \psi^D) = \sum_{\widetilde\varphi_{\mathfrak j}} W^{D}_{\widetilde{\pi} (f') \widetilde\varphi_{\mathfrak j}} (1) \overline {\widetilde W^{D}_{\widetilde\varphi_{\mathfrak j}} (1)},
\end{align*}
with  $\{ \widetilde \varphi_{\mathfrak j} \}$  an orthonormal basis of $V_{\widetilde\pi}$; here for   $\widetilde\varphi \in V_{\widetilde\pi}$
\begin{align*}
	\widetilde\pi (f') \widetilde\varphi & = \int_{\widetilde S(\BA)} f'(g) \widetilde \pi (g) \widetilde\varphi \, d g, \\
	W^D_{\widetilde\varphi} (1) & = \int_{\BA/F} \widetilde \varphi (\widetilde n(x)) \psi^D (-x) d x.
\end{align*}
The distributions $I_{\pi} (f, \psi)$ and $J_{\widetilde{\pi}} (f', \psi^D)$ are the contributions from $\pi$ and $\widetilde \pi$ to $I (f, \psi)$ and  $J (f', \psi^D)$ in their  spectral decompositions, respectively.

Recall that, if $\pi_{\varv}$ is unramified, there is a vector $\varphi_{\ro, \shskip \varv}$ that is fixed
under the action of $G(O_{\varv})$. For each Hecke function $f_{\varv}$ in $ H(Z(\Fv) \backslash G(\Fv))$ there is a constant $\hat f_{\varv} (\pi_{\varv})$ such that
\begin{equation}\label{5eq: f phi, pi}
	\pi_{\varv} (f_{\varv} ) \varphi_{\ro, \shskip \varv} = \hat f_{\varv} (\pi_{\varv}) \varphi_{\ro, \shskip \varv}.
\end{equation}
Similarly, if $\widetilde{\pi}_{\varv}$ is unramified, let $\widetilde\varphi_{\ro, \shskip \varv}$ be a vector that is fixed under $S(O_{\varv})$, then for each  $f_{\varv}'$ in $H (\widetilde{S} (\Fv))$ there is a constant $\hat f'_{\varv} (\widetilde\pi_{\varv})$ with
\begin{equation}\label{5eq: f' phi, pi tilde}
	\widetilde\pi_{\varv} (f'_{\varv} ) \widetilde\varphi_{\ro, \shskip \varv} = \hat f'_{\varv} (\widetilde\pi_{\varv}) \widetilde\varphi_{\ro, \shskip \varv}.
\end{equation}

\begin{thm}\label{thm:Jacquet:Main}
	Let $\pi$ be a cuspidal representation of $G (\BA)$ with trivial central character such that the distribution $I_{\pi} (f, \psi)$ is nontrivial.
	
	{\rm(1).} There is a unique cuspidal representation $\widetilde \pi$ of  $\widetilde{S}$ such that if $f$ and $f'$ match as in Theorem {\rm\ref{thm: I=J}} then
	\begin{equation}\label{5eq: I pi = J pi tilde}
		I_{\pi} (f, \psi) = J_{\widetilde{\pi} } (f', \psi^D).
	\end{equation}
	
	{\rm(2).} Suppose that $S$ satisfies the condition in Theorem {\rm\ref{thm: I=J}} and contains all the places
	where $\pi_{\varv}$ or $\widetilde{\pi}_{\varv}$ is not unramified. For $\varv \notin S$, if the Hecke functions $f_{\varv}$ and $f'_{\varv}$ correspond to each other under the isomorphism {\rm\eqref{5eq: Hecke isom}} in Proposition {\rm\ref{prop: Jacquet 3}}, then
	\begin{align}\label{5eq: hat f = hat f'}
		\hat f_{\varv} ( \pi_{\varv}) = \hat f'_{\varv} (\widetilde\pi_{\varv}).
	\end{align}
	
	{\rm(3).} $\widetilde{\pi} = \Theta (\pi, \psi^D)$.
\end{thm}

\begin{rem}
	The distribution $I_{\pi} (f, \psi)$ is nontrivial if and only if $L (\pi, 1/2) \neq 0$.
\end{rem}

In this theorem, (1) and (2) are due to Jacquet, while (3) is proven in \cite{BaruchMao-Global} by combining the work of Jacquet and the theory of Waldspurger (Theorem \ref{thm: Waldspurger-Local} and \ref{thm:Waldspurger-Global}).

\section{Review of the Local Bessel Distributions and Their Bessel Identities}\label{sec: local Bessel identity}

In this section, we review the local Bessel identities established in \cite{BaruchMao-NA}, \cite{BaruchMao-Real} and \cite{Chai-Qi-Bessel} which complement the global identiy \eqref{5eq: I pi = J pi tilde} in Theorem \ref{thm:Jacquet:Main}.
We shall retain the notations and assumptions as in Theorem \ref{thm:Jacquet:Main}.

\subsection{The local relative Bessel distribution $I_{\pi,\shskip\varv}(f_\varv, \psi_\varv)$.}

As in \S \ref{sec: defn of d}, we fix a choice of the local $\psi_{\varv}$-Whittaker functional $L_\varv$ on $\pi_\varv$, and use it to define the Hermitian form on $V_{\pi_\varv}$. For each $\varv\in S$, fix an orthonormal basis $\{\varphi_{\mathfrak{i}, \shskip \varv} \}$ of $V_{\pi_\varv}$. For $\varv\notin S$, let $\varphi_{\ro,\shskip\varv}$ be the normalized spherical vector as in \S \ref{sec: defn of d}. In view of \eqref{2eq: c2}, after rescaled by the factor $c_2(\pi, S, \psi, \{L_\varv\})^{-1}$, the tensor products formed by these $\varphi_{\mathfrak{i}, \shskip \varv}$ and $\varphi_{\ro,\shskip\varv}$ give rise to an orthonormal basis for $ V_{\pi}^S $, say, denoted by  $\{ \varphi_{\mathfrak{i}(S)} \}$.
With our choice of $f$, the operator $\pi(f)$ preserves $V_{\pi}^S$ but annihilates its orthogonal complement in $V_\pi$. Then the expression in \eqref{5eq: defn of I pi} becomes
\begin{align}\label{6eq: I(f) = ...}
I_{\pi}(f,\psi)=\sum_{\varphi_{\mathfrak{i}(S)}}Z(\pi(f)\varphi_{\mathfrak{i}(S)})\overline {W_{\varphi_{\mathfrak{i}(S)}}(1)}.
\end{align}

Using the Hecke theory for $\GL_2$, it is easy to verify the following lemma.
\begin{lem}
For a pure tensor $\varphi \in V_{\pi}^S$, we have
\begin{align}\label{6eq: Z(phi)}
Z(\varphi)=c_1(\pi, S,\psi, \{L_\varv\})L(\pi, 1/2)\prod_{\varv\shskip\in S}P_{\varv}(\varphi_\varv),
\end{align}
with
\begin{align}\label{6eq: P(phi)}
P_\varv(\varphi_\varv)= \frac 1 {{L(\pi_{\varv}, s)}} \int_{F_{\varv}^\times}L_\varv(\pi_\varv(t(a))\varphi_\varv)|a|_\varv^{s-1/2}\dx a \,  \bigg|_{\hskip -1 pt s=1/2}.
\end{align}
\end{lem}

We now define the local relative Bessel distribution as in \cite{BaruchMao-NA,BaruchMao-Real,Chai-Qi-Bessel},
\begin{align}\label{6eq: I(f), local}
I_{\pi_\varv}(f_\varv,\psi_\varv)= \sum_{\varphi_{\mathfrak{i}, \shskip\varv}} P_\varv(\pi_\varv(f_\varv)\varphi_{\mathfrak{i}, \shskip\varv})\overline {L_\varv(\varphi_{\mathfrak{i}, \shskip\varv})}.
\end{align}
Note that this definition is independent on the choice of $  L_\varv$ at the beginning.

%where the sum is taken over the orthonormal basis of $V_{\pi, \varv}$.

The following proposition is readily established on \eqref{2eq: c1}, \eqref{2eq: defn of d} and \eqref{6eq: I(f) = ...}-\eqref{6eq: I(f), local}.
\begin{prop}\label{prop: express of I}
Let notations be as above. We have
\begin{align}
I_\pi(f,\psi)=L(\pi, 1/2)|d_\pi(S, \psi)|^2\prod_{\varv\shskip\in S}I_{\pi_\varv}(f_\varv,\psi_\varv)\prod_{\varv\shskip\notin S}\hat {f}_\varv(\pi_\varv).
\end{align}
\end{prop}

\subsection{The local Bessel distribution $J_{\widetilde \pi_\varv}(f_\varv ',\psi_\varv^D)$.}

We can apply the same argument to the distribution $J_{\widetilde \pi}(f', \psi^D)$. The corresponding local Bessel distribution is defined by
\begin{align}\label{6eq: defn of J, local}
J_{\widetilde \pi_\varv}(f_\varv ', \psi_\varv^D)=\sum_{\widetilde \varphi_{\mathfrak{j}, \shskip\varv}}\widetilde L_{\varv}^D(\widetilde \pi_\varv(f_\varv ')\widetilde \varphi_{\mathfrak{j},\shskip\varv})\overline {\widetilde L_{\varv}^D(\widetilde \varphi_{\mathfrak{j},\shskip\varv})}.
\end{align}
where the sum is over an orthonormal basis $\{\widetilde \varphi_{\mathfrak{j},\shskip\varv} \}$ of $V_{\widetilde \pi_\varv}$. Again this definition is independent on the local $\psi^D_{\varv}$-Whittaker functional $\widetilde L_{\varv}^D$ that we choose.

\begin{prop}\label{prop: express of J} Let notations be as above. We have
\begin{align}
J_{\widetilde \pi}(f', \psi^D)=|d_{\shskip\widetilde \pi}(S, \psi^D)|^2\prod_{\varv\shskip\in S}J_{\widetilde \pi_\varv}(f_\varv ', \psi_\varv^D)\prod_{\varv\shskip\notin S}\hat {f}'_\varv(\widetilde \pi_\varv).
\end{align}
\end{prop}

\subsection{The local Bessel identity}

\begin{thm}\label{thm:Local Bessel Identity}
Let $\pi_\varv$ be an irreducible unitary representation of $G(F_\varv)$ with trivial central character. Let $D \in F_{\varv}^{\times} $. Put $\widetilde \pi_\varv =\theta(\pi_\varv, \psi_\varv^D)$. Suppose that  $f_{\varv} \in C_c^{\infty} (Z(F_{\varv} ) \backslash G(F_{\varv} ) )$ and $f'_{\varv} \in C_c^{\infty} (\widetilde{S} (F_{\varv}))$ match as in Proposition {\rm\ref{prop: Jacquet 2}}. We have
\begin{align}
J_{\widetilde \pi_\varv}(f'_\varv, \psi_\varv^D)=|2D|_\varv \shskip \epsilon (\pi_\varv, 1/2)L(\pi_\varv, 1/2)I_{\pi_\varv}(f_\varv, \psi_\varv),
\end{align}
in which $J_{\widetilde \pi_\varv}(f'_\varv, \psi_\varv^D)$ and $I_{\pi_\varv}(f_\varv, \psi_\varv)$ are defined as in {\rm\eqref{6eq: I(f), local}} and {\rm\eqref{6eq: defn of J, local}} respectively.
\end{thm}

\begin{rem}\label{rmk:Local Bessel Identity}
The proof of this theorem is quite technical.
 The identity is established in \cite{BaruchMao-NA,BaruchMao-Real} and \cite{Chai-Qi-Bessel} when $\varv$ is non-Archimedean, real and complex, respectively. In the former two papers of Baruch and Mao, the complexity lies mostly in the representation theory aspect, although the real case is analytic in nature. In the complex case, however, the representation theory is the simplest---$\widetilde{\SL}_2 (\BC)$ is just the trivial double cover of ${\SL}_2 (\BC)$---while the difficulty is to prove the complex analogue of the exponential integral formulae of Weber and Hardy---we now have to integrate twice, radially and angularly. The proof is done in two papers of the second author \cite{Qi-Sph,Qi-II-G}, and it involves a number of classical formulae for special functions {\rm(}Bessel and Kummer{\rm)}, quite a lot of combinatorial identities {\rm(}the combinatorial theory of Bessel functions therein should be of some independent interest{\rm)}, the method of stationary phase and a little bit of differential equation theory.
\end{rem}

	\section{Proof of Theorem \ref{thm:Waldspurger-Basic} and Theorem \ref{thm:Waldspurger-Twist}}\label{sec: proofs}
	
%	Following \cite{BaruchMao-Global}, we now outline the proofs our main results.
	\subsection{Proof of  Theorem \ref{thm:Waldspurger-Basic}}
%\begin{proof}[Proof of  Theorem \ref{thm:Waldspurger-Basic}]
		Let $\pi$ and $\widetilde \pi=\Theta(\pi, \psi^D)$ be as in the theorem. Let $f=\otimes f_\varv$ and $f'=\otimes f'_\varv$ %such that, if $\varv\notin S$, $f_\varv$ is a Hecke function and $f'_\varv=\eta_\varv(f_\varv)$, and if $\varv\in S$, $f_\varv$ matches $f'_\varv$. Then $f$ and $f'$
		match as in Theorem \ref{thm: I=J}. By Theorem \ref{thm:Jacquet:Main} (1, 3), we have
		\begin{align*}
		I_\pi(f,\psi)=J_{\widetilde \pi}(f', \psi^D).
		\end{align*}
		By Proposition \ref{prop: express of I} and \ref{prop: express of J}, this equality may be explicitly written as
		\begin{align*}
		L(\pi, 1/2)|d_\pi(S, \psi)|^2\prod_{\varv\shskip\in S}I_{\pi_\varv}(f_\varv,\psi_\varv)\prod_{\varv\shskip\notin S}\hat {f}_\varv(\pi_\varv)=|d_{\shskip\widetilde \pi}(S, \psi^D)|^2\prod_{\varv\shskip\in S}J_{\widetilde \pi_\varv}(f_\varv ', \psi_\varv^D)\prod_{\varv\shskip\notin S}\hat {f}'_\varv(\widetilde \pi_\varv).
		\end{align*}
		It then follows from Theorem \ref{thm:Jacquet:Main} (2) and \ref{thm:Local Bessel Identity} that
		\begin{align*}
		L(\pi, 1/2)|d_\pi(S, \psi)|^2=|d_{\shskip\widetilde \pi}(S, \psi^D)|^2\prod_{\varv\shskip \in S} |2D|_\varv\epsilon(\pi_\varv, 1/2)L(\pi_\varv, 1/2).
		\end{align*}
		As $|2D|_\varv=1$ for $\varv\notin S$, we get $\prod_{\varv\shskip\in S}|2D|_\varv=1$. As $\epsilon(\pi_\varv, 1/2)=1$ for $\varv\notin S$ and $\epsilon(\pi, 1/2)=1$, we get $\prod_{\varv\shskip\in S}\epsilon(\pi_\varv, 1/2)=1$. Thus follows the identity in the theorem.
%	\end{proof}

\subsection{Proof of  Theorem \ref{thm:Waldspurger-Twist}}
%\begin{proof}[Proof of  Theorem \ref{thm:Waldspurger-Twist}]
Let $\pi$ and $\widetilde \pi^\epsilon$ be as in the \S \ref{sec: Waldspurger, twisted}. Assume $D\in F^{\epsilon_{\oldstylenums{0}}}(\pi)$.

We first consider the case $\epsilon_{\oldstylenums{0}} \neq \epsilon$. Then $\epsilon(D, \pi_\varv)=\epsilon_{\oldstylenums{0}, \shskip \varv}\neq \epsilon_\varv$ for some $\varv\in \varSigma$. Thus $\theta(\pi_\varv\otimes \chiup_D, \psi_\varv^D)\neq \widetilde \pi^\epsilon_\varv$. By Theorem \ref{thm: Waldspurger-Local}, $\widetilde \pi_\varv^{\epsilon}$ does not have a nontrivial $\psi_\varv^D$-Whittaker functional, which then implies  that $\widetilde \pi^\epsilon$ does not have a nontrivial $\psi^D$-Whittaker model. Hence $d_{\widetilde \pi^{\epsilon}}(S, \psi^D)= 0$ by definition.

Let $S_D$ be a finite set of places such that $|D|_\varv=1$ if $\varv\notin S_D$. Put $S_1=S\cup S_D$, $S_2=S_1\smallsetminus S$.

We now consider  $\Theta(\pi\otimes \chiup_D, \psi^D)$. First of all, by Theorem \ref{thm:Waldspurger-Global} (5), $\Theta(\pi\otimes \chiup_D, \psi^D) $ is equal to either $\widetilde \pi^{\epsilon_{\oldstylenums{0}}}$ or $0$.

Suppose $\Theta(\pi\otimes \chiup_D, \psi^D)=0$. Then $L(\pi\otimes \chiup_D, 1/2)=0$. By Theorem \ref{thm:Waldspurger-Global} (2, 3), $\widetilde \pi^{\epsilon_{\oldstylenums{0}}}$ does not have a $\psi^D$-Whittaker model; otherwise we would have $\Theta (\widetilde\pi^{\epsilon_{\oldstylenums{0}}}, \psi^D) \otimes \chiup_D = \pi$ and therefore $ \widetilde\pi^{\epsilon_{\oldstylenums{0}}} = \Theta (\pi\otimes \chiup_D, \psi^D) $. Hence $d_{\shskip \widetilde \pi^{\shskip\epsilon_{\oldstylenums{0}}}}(S, \psi^D)=0$. Thus the identity \eqref{4eq: Waldspurger twisted} holds in this case.

Next we consider the case $\Theta(\pi\otimes \chiup_D, \psi^D)=\widetilde \pi^{\epsilon_{\oldstylenums{0}}}$. For $\varv\notin S_1$, $\pi\otimes \chiup_D$ is unramified at $\varv$. Moreover, as $\psi_\varv$ is unramified, the representation $\theta(\pi_\varv\otimes \chiup_D, \psi_\varv)$ is also unramified. We apply Theorem \ref{thm:Waldspurger-Basic} and get
\begin{align*}
|d_{\pi\otimes \chiup_D}(S_1, \psi)|^2L^{S_1}(\pi\otimes \chiup_D, 1/2)=|d_{\shskip \widetilde \pi^{\shskip\epsilon_{\oldstylenums{0}}}}(S_1, \psi^D)|^2.
\end{align*}
According to Lemma 7.1 in \cite{BaruchMao-Global} , we have $d_{\pi\otimes \chiup_D}(S_1, \psi)=d_\pi(S_1, \psi)$ and hence
\begin{align*}
|d_\pi(S_1, \psi)|^2L^{S_1}(\pi\otimes \chiup_D, 1/2)=|d_{\shskip \widetilde \pi^{\shskip\epsilon_{\oldstylenums{0}}}}(S_1, \psi^D)|^2.
\end{align*}
Take vectors $   \varphi=\otimes \,   \varphi_{ \varv} \in   V_{\pi}  $ and $\widetilde \varphi=\otimes \, \widetilde \varphi_{ \varv} \in V_{ \widetilde \pi^{\shskip\epsilon_{\oldstylenums{0}}} }$, with  $  \varphi_{ \varv}=  \varphi_{\ro,\shskip \varv}$ and $\widetilde \varphi_{ \varv}=\widetilde \varphi_{\ro,\shskip \varv}$ for $\varv\notin S$, then by \eqref{2eq: explicit d} and \eqref{2eq: explicit d tilde} we have
\begin{align*}
|d_\pi(S_1, \psi)|^2=|d_\pi(S, \psi)|^2\prod_{\varv\in S_2}e(\varphi_{\ro,\shskip \varv}, \psi_\varv)
\end{align*}
and
\begin{align*}
|d_{\shskip\widetilde \pi^{\shskip\epsilon_{\oldstylenums{0}}}}(S_1, \psi^D)|^2=|d_{\widetilde \pi^{\epsilon_{\oldstylenums{0}}}}(S, \psi^D)|^2\prod_{\varv\shskip \in S_2}e(\widetilde \varphi_{\ro,\shskip \varv}, \psi_{\varv}^D),
\end{align*}
where the constants $e(\varphi_{\ro,\shskip \varv}, \psi_\varv), e(\widetilde \varphi_{\ro,\shskip \varv}, \psi_{\varv}^D)$ are defined by \eqref{2eq: local e(phi)} and \eqref{2eq: local e(tilde phi)}, respectively.
It then follows that
\begin{align}\label{7eq: d pi = d tpi e/e}
|d_\pi(S, \psi)|^2L^{S_1}(\pi\otimes \chiup_D, 1/2)=|d_{\shskip\widetilde \pi^{\shskip\epsilon_{\oldstylenums{0}}}}(S, \psi^D)|^2\prod_{\varv\shskip\in S_2}\frac{e(\widetilde \varphi_{\ro, \varv}, \psi_{\varv}^D)}{e(\varphi_{\ro, \varv}, \psi_\varv)}.
\end{align}
For $\varv\in S_2$, $\pi_\varv$ is unramified and unitary, so $\widetilde \pi^{\epsilon_{\oldstylenums{0}}}_\varv=\theta(\pi_\varv\otimes \chiup_D, \psi^D_\varv)=\theta(\pi_\varv, \psi_\varv)$ by Theorem \ref{thm: Waldspurger-Local} (1).
At this point, we invoke Proposition 8.1 and 8.2 in \cite{BaruchMao-Global} as follows.
\begin{prop*}
	Let $\varv$ be a non-Archimedean place, with
	odd residue characteristic. Suppose that $\psi_{\varv}$ is unramified and that $|D|_{\varv} = 1$ or $q_{\varv}\-$. Let $\pi_{\varv}$ be an unramified unitary representation of $G (F_{\varv})$. Let $  \varphi_{\ro,\shskip \varv}$ and $\widetilde \varphi_{\ro,\shskip \varv}$ be the unramified vectors of $\pi_{\varv}$ and $\widetilde \pi_\varv = \theta(\pi_\varv, \psi_\varv) $ defined as in \S  {\rm\ref{sec: Fourier coefficients, d}}, respectively. Then
	 \begin{align}\label{7eq: quotient of e, unramified}
	 \frac{e(\widetilde \varphi_{\ro,\shskip \varv}, \psi_{\varv}^D)}{e(\varphi_{\ro,\shskip \varv}, \psi_\varv)}= \frac 1 { |D|_\varv L(\pi_\varv\otimes \chiup_D, 1/2)}.
	 \end{align}
\end{prop*}

As $D$ is square-free, we may apply the formula \eqref{7eq: quotient of e, unramified} for $\varv \in S_2$.
Since $S_1=S\cup S_2$, it follows from \eqref{7eq: d pi = d tpi e/e} and \eqref{7eq: quotient of e, unramified} that
\begin{align}
|d_\pi(S, \psi)|^2L^{S}(\pi\otimes \chiup_D, 1/2)=|d_{\widetilde \pi^{\epsilon_{\oldstylenums{0}}}}(S, \psi^D)|^2 / |D|_{S_2}.
\end{align}
Finally note that $|D|_\varv=1$ for $\varv\notin S_1$ and hence $|D|_{S_1} = 1$. So $1/|D|_{S_2}=|D|_{S}$. We get \eqref{4eq: Waldspurger twisted} and the theorem.
%\end{proof}

	\section{Metaplectic Ramanujan Conjecture and Lindel\"of Hypothesis}\label{sec: Ramanujan-Lindelof}

In this final section, we discuss the connection between the metaplectic Ramanujan conjecture and the Lindel\"of hypothesis.

\subsection{Statement of conjectures}

Let $\widetilde \pi \in \widetilde A_{0\shskip 0}$. %be an irreducible cuspidal automorphic representation of $\widetilde{\SL}_2 (\BA)$.
Suppose that $S$ is a finite set of places of $F$ containing all bad places along with places $\varv$ where $\widetilde \pi_\varv$ is not unramified. For $S_0\subset S$ and  $\widetilde \varphi\in V_{\widetilde \pi }^S$, %$D$ a square-free integer,
 we define
\begin{align}
d_{\shskip\widetilde \pi}(\widetilde \varphi, S_0, \psi^D)=\frac{\big|\widetilde W^D_{\widetilde \varphi}(1)\big|}{\|\widetilde \varphi\|}\prod_{\varv\shskip \in S_0}\frac{\|\widetilde \varphi_\varv\|_D}{|\widetilde L_\varv^D(\widetilde \varphi_\varv)|}.
\end{align}
This constant is well defined and independent on the choice of the $\psi^D$-Whittaker functional ${\widetilde L_\varv^D}$. Moreover, in view of \eqref{2eq: explicit d tilde}, we also have
\begin{align}\label{8eq:tilde d}
d_{\shskip\widetilde \pi}(\widetilde \varphi, S_0, \psi^D)=d_{\shskip\widetilde \pi}(S, \psi^D)\prod_{\varv \shskip\in S \smallsetminus S_0}\frac{|\widetilde L_\varv^D(\widetilde \varphi_\varv)|}{\|\widetilde \varphi_\varv\|_D}.
\end{align}

Recall that $S_\infty$ is the set of Archimedean places of $F$ and that $|D|_{S_\infty}=\prod_{\varv \shskip \in S_\infty}|D|_\varv$ for $D\in F^\times$.

We can now state the metaplectic Ramanujan conjecture as follows.

\begin{conj}[Metaplectic Ramanujan conjecture]\label{conj: Ramanujan}
	Let $\widetilde \pi $ be an irreducible cuspidal automorphic representation of $\widetilde{\SL}_2 (\BA)$ in $\widetilde A_{0\shskip 0}$. Let $\widetilde{\varphi}$ be a cusp form in $V_{\widetilde{\pi} }$. Let $D$ be a square-free integer in $F^{\times}$. For all $\alpha>0$, we have
	\begin{align}\label{8eq: ineq Ramanujan}
	\big|d_{\widetilde \pi}(\widetilde \varphi, S_\infty, \psi^D) \big|\lll |D|_{S_\infty}^{\alpha-\frac 1 2}
	\end{align}
	as $|D|_{S_\infty}\to \infty$, where the implied constant depends only on $\widetilde \pi, \widetilde \varphi$ and $\alpha$. Equivalently, for all $\alpha>0$, we have
	\begin{align}\label{8eq: ineq Ramanujan2}
	\big|\widetilde W^D_{\widetilde \varphi}(1)\big| \prod_{ \varv \shskip \in S_{\infty} }  e (\widetilde{\varphi}_{\varv}, \psi_{\varv}^D )^{\frac 1 2}  \lll |D|_{S_\infty}^{\alpha-\frac 1 2}
	\end{align}
	as $|D|_{S_\infty}\to \infty$. %Here $e (\widetilde{\varphi}_{\varv}, \psi_{\varv}^D )$ defined as in {\rm\eqref{2eq: local e(tilde phi)}}.
\end{conj}

In the classical language, the constant $d_{\widetilde \pi}(\widetilde \varphi, S_\infty, \psi^D)$ is indeed a renormalized $D$-th Fourier coefficient of the cusp form $\widetilde \varphi$. Details may be found for example in \cite[III]{Waldspurger-Formula}, \cite[\S 9]{BaruchMao-Global} and \cite[\S 2]{Baruch-Mao-KZ-Maass}.

Next we state a special case of the Lindel\"of hypothesis, which is a conjecture on the bound of central value of twisted $L$-functions for $\PGL_2 $.

\begin{conj}[Lindel\"of hypothesis]\label{conj: Lindelof}
	Let $\pi$ be an irreducible cuspidal automorphic representation of $\PGL_2 (\BA)$.
     Let $D$ be a square-free integer in $F^{\times}$. For all $\beta>0$, we have
     \begin{align}\label{8eq: ineq Lindelof}
     \big|L^{S_\infty}(\pi\otimes \chiup_D, 1/2) \big|\lll |D|_{S_\infty}^\beta
     \end{align}
     as $|D|_{S_\infty} \to  \infty$, where the implied constant depends only on $\pi$ and $\beta$.	
\end{conj}

\subsection{The Ramanujan-Lindel\"of equivalence}

\begin{thm}\label{thm: R-L equivalence}
	The inequality {\rm\eqref{8eq: ineq Ramanujan}} holds for $\alpha > 0$ {\rm(}and for all
	$\widetilde{\pi}$ and $\widetilde{\varphi}$ as in Conjecture {\rm\ref{conj: Ramanujan})} if and only if the inequality {\rm\eqref{8eq: ineq Lindelof}} holds for $\beta = 2 \alpha > 0$ {\rm(}and for all $\pi$ as in Conjecture {\rm\ref{conj: Lindelof})}.
	In particular, Conjecture {\rm\ref{conj: Ramanujan}}
	is equivalent to Conjecture {\rm\ref{conj: Lindelof}}.
\end{thm}

\begin{proof}
	Let $ \pi \in A_{0,\shskip i} $ and $\widetilde{\pi} \in  \widetilde A_{0\shskip 0}$ be given such that $ \pi = S_{\psi } (\widetilde {\pi})$. %Suppose $\widetilde {\pi} =  \widetilde \pi^{\shskip\epsilon_{\oldstylenums{0}}}$.
	Restrict $D$ to range in the set of square-free integers in $F^{\times}$ with $    \Theta (\pi\otimes \chiup_D, \psi^D) = \widetilde{\pi} $ so that  $ d_{\shskip\widetilde \pi}(S, \psi^D) $ and $ L(\pi\otimes \chiup_D, 1/2) $ are nonzero and that  the identity \eqref{4eq: Waldspurger twisted} in  Theorem \ref{thm:Waldspurger-Twist} is nontrivial.
	We then rewrite the identity (\ref{4eq: Waldspurger twisted}) as follows,
	\begin{align}\label{8eq: identity}
	\big|d_{\shskip\widetilde \pi}(\widetilde \varphi, S_\infty, \psi^D)  \big|^2 |D|_{S_{\infty}} \prod_{\varv \shskip\in S \smallsetminus S_\infty} C_{\varv}  (D, \pi_\varv, \widetilde{\varphi}_\varv)  =\big|d_\pi(S, \psi)\big|^2 L^{ S_{\infty}}(\pi\otimes \chiup_D, 1/2),
	\end{align}
	with
	\begin{align}\label{8eq: Cv}
C_{\varv} (D, \pi_\varv, \widetilde{\varphi}_\varv) =   |D|_{\varv} L (\pi_{\varv}\otimes \chiup_D, 1/2) \cdot 	 {\|\widetilde \varphi_\varv\|_D^2} / {|\widetilde L_\varv^D(\widetilde \varphi_\varv)|^2}.
	\end{align} Here we have used the identity \eqref{8eq:tilde d}.
	
	For each $\varv \in S \smallsetminus S_{\infty}$, since $D$ is square-free, it is clear that $|D|_{\varv}$ and $L (\pi_\varv\otimes \chiup_D, 1/2)$ take only finitely many possible nonzero values, depending on $\pi_\varv$ (in the latter case). It is also proven in \cite[Lemma 7.2]{BaruchMao-Global} that there
	are only finitely many possible nonzero values of $ {\|\widetilde \varphi_\varv\|_D^2} / {|\widetilde L_\varv^D(\widetilde \varphi_\varv)|^2}$ when $D$ varying in the set of square-free integers. Combining these, we have
	\begin{align}\label{8eq: C = 1}
	  |C_{\varv} (D, \pi_\varv, \widetilde{\varphi}_\varv) | \asymp_{ \, \pi_\varv, \shskip \widetilde{\varphi}_\varv } 1.
	\end{align}
	The asymptotic notation here means that there are positive constants $A = A (\pi_\varv, \widetilde{\varphi}_\varv)$ and $B = B (\pi_\varv, \widetilde{\varphi}_\varv)$, depending only on $ \pi_\varv $ and $ \widetilde{\varphi}_\varv$, such that  $B< |C_{\varv} (D, \pi_\varv, \widetilde{\varphi}_\varv) | < A$.
	
Now we turn to the proof of our theorem.

Assume first that (\ref{8eq: ineq Ramanujan}) holds for   $\alpha>0$. Given $\pi$, let $S$ be a fixed finite set of places so that (\ref{4eq: Waldspurger twisted}) holds. For each $\widetilde{\pi}^{\epsilon} \in S^{-1}_{\psi} (\pi)$, fix a cusp form $\widetilde{\varphi}^{\epsilon} \in V_{\widetilde{\pi}^{\epsilon}}$. Applying (\ref{8eq: ineq Ramanujan}) for these $\widetilde{\pi}^{\epsilon} $ and $ \widetilde{\varphi}^{\epsilon} $, along with \eqref{8eq: identity}-\eqref{8eq: C = 1}, we get
\begin{align*}
L^{ S_{\infty}}(\pi\otimes \chiup_D, 1/2) \lll \big|d_{\shskip\widetilde \pi}(\widetilde \varphi, S_\infty, \psi^D)  \big|^2 |D|_{S_{\infty}} \lll |D|_{S_{\infty}}^{2 \alpha}, \hskip 15 pt |D|_{S_\infty} \to  \infty,
\end{align*}
for all square-free $D$ in $F^{\epsilon} (\pi) $ (note that if  $D \in F^{\epsilon} (\pi) $ but  $    \Theta (\pi\otimes \chiup_D, \psi^D) = 0 $ then $ L (\pi\otimes \chiup_D, 1/2) = 0$ and the inequality above holds trivially).  This proves \eqref{8eq: ineq Lindelof} with $\beta = 2 \alpha$. Note that the implied constant in each step depends (ultimately) only on $\pi$ and $\alpha$.

The reverse direction may be proven similarly.  Assume that (\ref{8eq: ineq Lindelof}) holds for   $\beta=2\alpha>0$. Given $\widetilde{\pi}$, let $\pi = S_{\psi} (\widetilde{\pi})$ and let  $S$ be fixed as before.  Assume $D$ is such that $    \Theta (\pi\otimes \chiup_D, \psi^D) = \widetilde{\pi} $, as otherwise $ d_{\shskip\widetilde \pi}(S, \psi^D) = 0 $ and hence $ d_{\shskip\widetilde \pi}(\widetilde \varphi, S_\infty, \psi^D) = 0 $. Applying (\ref{8eq: ineq Lindelof}) for $\pi$, along with \eqref{8eq: identity}-\eqref{8eq: C = 1}, we get
\begin{align*}
	\big|d_{\widetilde \pi}(\widetilde \varphi, S_\infty, \psi^D) \big|\lll \big|L^{ S_{\infty}}(\pi\otimes \chiup_D, 1/2) |D|_{S_{\infty}}^{-1} \big|^{\frac 1 2} \lll |D|_{S_{\infty}}^{\alpha - \frac 1 2}, \hskip 15 pt |D|_{S_\infty} \to  \infty,
	\end{align*}
	as desired. The implied constant depends only on $\widetilde \pi, \widetilde \varphi$ and $\alpha$.
\end{proof}

\subsection{Nontrivial bound toward the metaplectic Ramanujan conjecture}

In view of \eqref{1eq: subconvexity} and \eqref{1eq: subconvexity, D}, the inequality \eqref{8eq: ineq Lindelof} holds for any $\beta > 2 \widetilde{\theta} = \text{\Large$\frac 3 {8} \hskip -1 pt $} + \text{\Large$ \hskip -1 pt \frac {1} {4}$}\theta$. As a consequence of Theorem \ref{thm: R-L equivalence}, we have the following theorem.

\begin{thm}\label{thm: Ramanujan}
	Let $\theta$ be any exponent toward the Ramanujan-Petersson conjecture   for $\GL_2 (\BA)$.
	Then the bound in {\rm\eqref{8eq: ineq Ramanujan}} holds for any $\alpha > \widetilde{\theta} = \text{\Large$\frac 3 {16} \hskip -1 pt $} + \text{\Large$ \hskip -1 pt \frac {1} {8}$}\theta$.
\end{thm}

Note that $\theta = 0 $ is the Ramanujan-Petersson conjecture while the Kim-Sarnak $\theta = \text{\Large$ \hskip -1 pt \frac {7} {64}$}$ is the best known exponent.

It is  suggested in \cite{Iwaniec-Z/2} that the bound in {\rm\eqref{8eq: ineq Ramanujan}} would hold trivially for any $ \alpha > \text{\Large$\frac 1 4$}$ while \eqref{8eq: ineq Lindelof} with $\beta > \text{\Large$\frac 1 2$}$ is the convexity bound.
The exponent in Theorem \ref{thm: Ramanujan} was obtained over the rational field $\BQ$ in \cite{BHM-Mao} and \cite{Baruch-Mao-KZ-Maass} for holomorphic modular forms and Maass forms  of half-integral weight respectively. Moreover, the Burgess-type exponent $\widetilde{\theta} = \text{\Large$\frac 3 {16}$}$ may be deduced from \cite{BH-Hybrid} and the Weyl-type $ \widetilde{\theta} = \text{\Large$\frac 1 {6}$} $ was achieved in \cite{PM-Cubic} in the holomorphic case. Our exponent also matches that in \cite{Blomer-Harcos-TR} in the totally real field case.  In comparison, the exponent of Iwaniec  and Duke is $\widetilde{\theta} = \text{\Large$\frac 3 {14}$}$.

	\begin{acknowledgement}
		We are grateful to Jim Cogdell for helpful commentaries and discussions on this work. We thank Gergely Harcos and the referee for bringing to our attention the Burgess and Weyl-type subconvexity bounds and the work of Qiu on the Whittaker period formula.
	\end{acknowledgement}
	
	\delete{
%	\section{Explicit formula of local factors}
	
	When translating the Waldspurger formula
	into the classical language, one must compute explicitly the local factors $e (\varphi_{\varv}, \psi)$ and $e( \widetilde \varphi_{ \varv}, \psi^D)$   for some specific choices of the vectors $\varphi$ and $\widetilde{\varphi}$. In \cite[\S 8]{BaruchMao-Global}, this is done  for
	\begin{itemize}
		\item [-] principal series, complementary series and special representations over a non-Archimedean local field with odd residue characteristic, and
		\item  [-] holomorphic discrete series over $\BR$.
	\end{itemize}
In this section, we shall record their results for completeness and do the computations for
\begin{itemize}
	\item [-] principal  and complementary series over $\BR$ and $\BC$.
\end{itemize}
	Subsequently, we shall drop the $\varv$ from our notations.
	
%	\subsection{The non-Archimedean case}
	
	Let $F$ be a non-Archimedean local field, with odd residue characteristic. For simplicity, we shall assume that $\psi$ has order $0$ and that either $|D| = 1$ or $|D| = q\-$.

%	\section{Maass forms over an imaginary quadratic field}
	
	\red{Let us consider this at the end. I don't know if it's of any interest, and it might be complicated.}
	
}
	
%	\bibliographystyle{alphanum}
%	\bibliography{references}

\def\cprime{$'$} \def\cprime{$'$}

\end{document}